\documentclass[a4paper,11pt]{article} 

\usepackage{amsmath, amscd, amsthm, amssymb, mathrsfs, color, framed, calc, hyperref} 
\usepackage[applemac]{inputenc}
\usepackage[all]{xy} 
\usepackage[T1]{fontenc} 
\usepackage{textcomp} 
\usepackage[a4paper, left=3cm, right=3cm, top=4cm, bottom=4cm]{geometry}

\DeclareMathOperator{\im}{Im}

\DeclareMathOperator{\Hom}{Hom}

\DeclareMathOperator{\tr}{Tr}

\DeclareMathOperator{\SO}{SO}
\DeclareMathOperator{\Spin}{Spin}
\DeclareMathOperator{\Cliff}{Cliff}

\DeclareMathOperator{\GL}{GL}

\DeclareMathOperator{\Rm}{Rm}
\DeclareMathOperator{\Ric}{Ric}
\DeclareMathOperator{\vol}{vol}

\newcommand{\Q}{\mathbb Q}
\newcommand{\R}{\mathbb R}

\newcommand{\Z}{\mathbb Z}
\newcommand{\N}{\mathbb N}

\newcommand{\diff}{\text{\rm d}}
\newcommand{\del}{\partial}

\newcommand{\dvol}{\mathrm{dvol}}

\newcommand{\tgk}{\tilde{g}_k}

\renewcommand{\mod}{\text{\rm mod}\,}

\renewcommand{\H}{\mathbb H}

\renewcommand{\div}{\mathrm{div}}

\theoremstyle{plain}
	\newtheorem{theorem}{Theorem}
	\newtheorem{proposition}[theorem]{Proposition}
	\newtheorem{lemma}[theorem]{Lemma}
	\newtheorem{corollary}[theorem]{Corollary}

\theoremstyle{definition}
	\newtheorem{definition}[theorem]{Definition}
	\newtheorem{remark}[theorem]{Remark}

\theoremstyle{plain}
	\newtheorem*{theorem*}{Theorem}
	\newtheorem*{proposition*}{Proposition}
	\newtheorem*{lemma*}{Lemma}
	\newtheorem*{corollary*}{Corollary}
	\newtheorem*{conjecture*}{Conjecture}
	\newtheorem*{assumption*}{Hypothesis which will lead to a contradiction}
\theoremstyle{definition}
	\newtheorem*{definition*}{Definition}
	\newtheorem*{remark*}{Remark}
	\newtheorem*{remarks*}{Remarks}
	
\numberwithin{equation}{section}

\numberwithin{theorem}{section}

\usepackage{authblk}
\title{Examples of compact Einstein four-manifolds with negative curvature}
\author{Joel Fine and Bruno Premoselli\\ Universit\'e libre de Bruxelles}
\date{}                     %% if you don't need date to appear
\newcommand*{\defeq}{\mathrel{\vcenter{\baselineskip0.5ex \lineskiplimit0pt
                     \hbox{\scriptsize.}\hbox{\scriptsize.}}}%
                     =}

\begin{document}

\maketitle

\vfill
\begin{abstract}
We give new examples of compact, negatively curved Einstein manifolds of dimension $4$. These are seemingly the first such examples which are not locally homogeneous. Our metrics are carried by a sequence of 4-manifolds $(X_k)$ previously considered by Gromov and Thurston \cite{gromov-thurston}. The construction begins with a certain sequence $(M_k)$ of hyperbolic 4-manifolds, each containing a totally geodesic surface $\Sigma_k$ which is nullhomologous and whose normal injectivity radius tends to infinity with $k$. For a fixed choice of natural number $l$, we consider the $l$-fold cover $X_k \to M_k$ branched along $\Sigma_k$. We prove that for any choice of $l$ and all large enough $k$ (depending on $l$), $X_k$ carries an Einstein metric of negative sectional curvature. The first step in the proof is to find an approximate Einstein metric on $X_k$, which is done by interpolating between a model Einstein metric near the branch locus and the pull-back of the hyperbolic metric from $M_k$. The second step in the proof is to perturb this to a genuine solution to Einstein's equations, by a parameter dependent version of the inverse function theorem. The analysis relies on a delicate bootstrap procedure based on $L^2$ coercivity estimates. \end{abstract}
\vfill

\newpage

\tableofcontents

\newpage

\section{Introduction}

\subsection{Compact Einstein manifolds with negative scalar curvature}

A Riemannian manifold $(M,g)$ is called \emph{Einstein} if $\Ric(g) = \lambda g$, for some $\lambda \in \R$. This article gives a new construction of compact Einstein 4-manifolds with $\lambda < 0$. To put our result in context, we recall the other currently known methods for constructing compact Einstein manifolds with $\lambda <0$.
\begin{enumerate}
\item \textbf{Locally homogeneous Einstein manifolds.} These are Einstein manifolds whose universal cover is homogeneous, i.e, acted on transitively by isometries. Negatively curved examples include hyperbolic and complex-hyperbolic manifolds.

\item \textbf{K\"ahler--Einstein metrics.} A compact K\"ahler manifold with $c_1<0$ carries a K\"ahler--Einstein metric with $\lambda <0$. This is due to Aubin \cite{aubin} and Yau \cite{yau}.

\item \textbf{Dehn fillings of hyperbolic cusps.} Given a finite volume hyperbolic $n$-manifold with cusps, one can produce a compact manifold by Dehn filling: each cusp is cut off at finite distance to produce a boundary component diffeomorphic to a torus $T^{n-1}$; this is then filled in by gluing $D^2 \times T^{n-2}$ along their common $T^{n-1}$ boundary. The resulting manifold depends on the choice of identification of these two copies of $T^{n-1}$. For appropriate choices, the Dehn filling carries an Einstein metric with $\lambda <0$. When $n=3$ this is due to Thurston \cite{thurston}. In this dimension the Einstein metric is in fact hyperbolic, putting us back in the locally homogeneous case described above. For $n \geq 4$, the Einstein metrics are no longer locally homogeneous.  In these dimensions the original idea is due to Anderson \cite{anderson} and was later refined by Bamler \cite{bamler} (see also the excellent exposition of Biquard \cite{biquard}). 
\end{enumerate}

Of these three constructions, only the first is known to produce Einstein metrics which are \emph{negatively curved}, i.e., with all sectional curvatures negative.

Any manifold $X$ produced by Dehn filling admits no metrics of negative curvature, at least in dimension $n \geq 4$. This follows from Preissman's theorem: if $X$ is a compact, negatively curved Riemannian manifold then any non-trivial abelian subgroup of $\pi_1(X)$ is isomorphic to $\Z$. For a Dehn filling, the core $T^{n-2}$ of the filling is incompressible, giving an abelian subgroup of $\pi_1(X)$ isomorphic to $\Z^{n-2}$. We also remark that the Einstein metrics on Dehn fillings found in \cite{anderson,bamler} explicitly have some positive sectional curvatures.

The situation for K\"ahler--Einstein metrics is less clear. These metrics are found via an abstract existence theorem for solutions to a complex Monge--Amp\`ere equation. Unfortunately, this result gives no direct way to describe the Einstein metrics and, in particular, to estimate the sectional curvatures. The only exception to this is when additional topological information implies that any K\"ahler--Einstein metric must actually be complex-hyperbolic and so locally homogeneous. This happens, for example, when a K\"ahler surface $X$ satisfies $c_1(X)^2 = 3c_2(X)$ (as was noted by Yau \cite{yau}). Besides complex-hyperbolic manifolds, there are no other known examples of negatively curved K\"ahler--Einstein metrics.

Given this paucity of examples, it is of great interest to find new constructions of Einstein metrics and, in particular, examples with negative curvature which are not just locally homogeneous. It is this question we address in this article. We find \emph{infinitely many compact 4-manifolds that carry negatively curved Einstein metrics, but that admit no locally homogeneous Einstein metrics.} 

The manifolds on which we find our new Einstein metrics were constructed by Gromov and Thurston \cite{gromov-thurston} in 1987; we briefly describe these manifolds next.

\subsection{Statement of the results}

In \cite{gromov-thurston} Gromov and Thurston investigated pinching for negatively curved manifolds of dimension $n \geq 4$. They showed that for any $\epsilon >0$, there exists a compact Riemannian $n$-manifold $(X,g)$ with sectional curvatures pinched by $-1-\epsilon < \mathrm{sec}(h) \leq -1$ and yet $X$ does not admit a hyperbolic metric. Note that the positively curved analogue of this statement is false: any compact Riemannian manifold with sectional curvatures pinched by $1 < \sec \leq 4$ is necessarily a quotient of the sphere \cite{berger,klingenberg,brendle-schoen}. 

A natural question, which motivated this work, is \emph{do Gromov and Thurston's manifolds carry Einstein metrics?} We answer this question positively, at least in dimension $4$. In order to state our main result we first construct a particular family of hyperbolic manifolds which belong to the class investigated in \cite{gromov-thurston}. The output of our construction is summarised in the following proposition. 

\begin{proposition}
\label{hyperbolic-sequence}
For each $n \in \N$, there exists a sequence $(M_k)$ of compact hyperbolic $n$-manifolds with the following properties.
\begin{enumerate}
\item
The injectivity radius $i(M_k)$ satisfies $i(M_k) \to \infty$ as $k \to \infty$.
\item
For each $k$, there is a nullhomologous totally-geodesic codimension-2 submanifold $\Sigma_k \subset M_k$. Moreover, the normal injectivity radius of $\Sigma_k$ satisfies $i(\Sigma_k,M_k) \geq \frac{1}{2}i(M_k)$. 
\item
There is a constant $C$, independent of $k$ such that for all sufficiently large $k$, the volume of $\Sigma_k$ with respect to the hyperbolic metric satisfies
\begin{equation}
\label{vol-bound-we-want}
\vol(\Sigma_k) \leq C\exp\left( \frac{n^2 - 3n + 6}{4} i(M_k) \right)
\end{equation}
\end{enumerate}
\end{proposition}
We will see that there are infinitely many such sequences $(M_k)$. Gromov and Thurston's original construction was based on sequences $(M_k)$ satisfying the first two properties here. The details are explained in \S\ref{spin-GT}.

Since $[\Sigma_k]=0$, given any fixed integer $l \geq 2$, there is a cyclic $l$-fold cover $X_k \to M_k$ branched along $\Sigma_k$. One way to see this is to take a hypersurface $H_k$ bounding $\Sigma_k$, cut $M_k$ along $H_k$ to produce a manifold-with-boundary $M'_k$; now take $l$ copies of $M'_k$ and glue their boundaries appropriately. It is important to note that  for a ramified cover constructed this way, the local ramification degree is $l$ at every ramification point. We also point out that $\Sigma_k$ may have many  connected components. 

The pull-back of the hyperbolic metric from $M_k$ to $X_k$ is singular along the branch locus, with cone angle $2\pi l$. The main result of this article is that, in dimension 4 at least,  the manifolds produced do indeed carry smooth Einstein metrics:

\begin{theorem}\label{main-theorem}
Fix  $l \geq 2$ and let $(M_k)$ denote a sequence of compact hyperbolic 4-manifolds satisfying the conclusions of Proposition~\ref{hyperbolic-sequence}. Let $X_k$ be the cyclic $l$-fold cover of $M_k$ branched along $\Sigma_k$. For all sufficiently large $k$ (depending on $l$), $X_k$ carries an Einstein metric of negative sectional curvature which is not locally homogeneous.
\end{theorem}

We remark that actually the manifolds $X_k$ carry no locally homogenous Einstein metrics whatsoever. In dimension~4, the only possible locally homogenous Einstein metrics on a compact manifold with infinite fundamental group are flat, hyperbolic and complex-hyperbolic~\cite{jensen}. In our particular situation it is relatively straightforward to rule out these possibilities. First, the Remark~3.6 of Gromov--Thurston's article \cite{gromov-thurston} explains that the branched covers $X_k$ never admit hyperbolic metrics. To rule out complex-hyperbolic metrics note first that the signature of $X_k$ is zero. (This follows from the fact that $[\Sigma_k]=0$; see, e.g., equation (15) in the article~\cite{hirzebruch} of Hirzebruch.) From this, Hirzebruch's signature theorem and Chern--Weil it follows that any anti-self-dual metric on $X_k$ is automatically conformally flat. Since complex-hyperbolic metrics are anti-self-dual but not conformally flat, they are excluded. Finally, the manifolds $X_k$ admit a negatively curved metric (constructed in \S\ref{approximate-solution}) and so, by Bieberbach's and Preissman's theorems, they cannot admit flat metrics.

Theorem \ref{main-theorem}, together with deep 4-dimensional rigidity results for Einstein metrics on compact manifolds, gives another way to see the manifolds $X_k$ do not admit negatively curved locally homogeneous Einstein metrics at all. Indeed, if a compact 4-manifold is either hyperbolic or complex-hyperbolic then the locally homogeneous metric is the only possible Einstein metric (up to overall scale). This was proved in the hyperbolic case by Besson, Courtois and Gallot \cite{besson-courtois-gallot}, whilst the complex-hyperbolic case is due to LeBrun \cite{lebrun}. 

An important feature of our construction is that it is the first occurrence, in the compact case, of a negatively curved Einstein metric which is not locally homogeneous. \emph{Non-compact} examples are relatively easy to find. There is an infinite dimensional family of Einstein deformations of the hyperbolic metric, found by Graham and Lee \cite{graham-lee}; when the deformation is small enough, the curvature remains negative. In 4~dimensions, a 1-parameter family of such deformations with an explicit formula was given by Pedersen \cite{pedersen}. It was recently observed by Cort\'es and Saha \cite{cortes-saha} that Pedersen's metrics are also negatively curved, even when far from the hyperbolic metric. Other explicit examples can also be found in \cite{cortes-saha}. 

As an aside, we mention that our examples also satisfy another curvature inequality, namely that $W^++\frac{R}{12}<0$ (i.e., it is a negative definite endomorphism of the bundle $\Lambda^+$ of self-dual 2-forms). Again, the only other known compact Einstein examples are hyperbolic and complex-hyperbolic. This inequality means that the twistor space carries a natural symplectic structure \cite{fine-panov}. It is also central to an alternative gauge-theoretic description of Einstein 4-manifolds \cite{fine-krasnov-panov}. There is an action functional defined on the space of all definite connections over a 4-manifold, whose critical points are precisely Levi-Civita connections for Einstein metrics (this is described in detail in \cite{fine-krasnov-panov}). The Einstein metrics which arise this way are those for which $W^++\frac{R}{12}$ is a definite endomorphism of~$\Lambda^+$. The examples provided here show that this alternative formulation of Einstein metrics covers more than just locally homogeneous metrics. 

\subsection{Brief outline of the proof}

The proof of Theorem~\ref{main-theorem} has two steps. The first, carried out in \S\ref{approximate-solution}, is similar in spirit to that of the tightly pinched Gromov--Thurston metrics. We smooth out the pull-back of the hyperbolic metric from $M_k$ to $X_k$ to find a metric which is approximately Einstein. The larger the injectivity radius, the better we can make the approximation. It is important to note, however, that our approximate Einstein metrics are \emph{not} the same as the tightly pinched metrics that Gromov and Thurston consider. Inside the normal injectivity radius we use a fixed model Einstein metric whose sectional curvatures satisfy $\mathrm{sec} \leq -c$ for some constant $0<c<1$ which depends only on $l$ (and not on $k$). The least negative sectional curvatures, $\mathrm{sec} = -c$, are attained at points on the branch locus. At large distances from the branch locus, the model is asymptotic to the pull-back of the hyperbolic metric. We interpolate between these two metrics at a distance which tends to infinity with $k$. This gives a metric $g_k$ on $X_k$ which is Einstein everywhere except for an annular region of large radius and fixed width in the tubular neighbourhood around each connected component of the branch locus. In these annular regions $g_k$ is close to Einstein, with error that tends to zero as $k$ tends to infinity.

The second step of the proof is to use the inverse function theorem to prove that for all large $k$, the approximately Einstein metric $g_k$ can be perturbed to give a nearby genuine Einstein metric. This new Einstein metric has sectional curvatures which are very close to those of $g_k$ and so, in particular, are also negative. The analysis involved turns out to be quite delicate. The fact that $g_k$ has negative sectional curvatures leads to the fact that the linearised Einstein equations (in Bianchi gauge) are invertible, with $L^2$-control. However, the volume and diameter of $X_k$ are rapidly increasing with $k$ and so weighted H\"older spaces, rather than Sobolev spaces, are seemingly the appropriate choice of Banach spaces in which to apply the inverse function theorem. Even with these spaces, however, we are unable to obtain control over the derivative in every direction. We circumvent this as follows. Since the metric $g_k$ is made by interpolating two Einstein metrics, the error is supported in a subset $S_k \subset X_k$. We can control the inverse of the derivative when it acts on sections supported in $S_k$. This, together with a careful application of the inverse function theorem, is enough to conclude  

We set the problem up in \S\ref{setting-up-IFT} and reduce it to the key estimate, giving control on the inverse of the linearised Einstein equations on sections supported in $S_k$. We prove this estimate in \S\ref{proof-of-key-estimate}. Starting from the uniform $L^2$ control given by the linearised Einstein equations we perform an involved bootstrap procedure and transform it into a weighted H\"older control. This relies on Carleman-type estimates for the Green's operator of the linearised equations and on a precise control of the volume of the branch locus $\Sigma_k$ provided by \eqref{vol-bound-we-want}.

We close this brief outline with a comment on dimension. The model Einstein metric exists in all dimensions $n \geq 4$ and gives approximately Einstein metrics on Gromov--Thurston manifolds for all $n \geq 4$. Unfortunately, the control of the volume of the branch locus provided by Proposition~\ref{hyperbolic-sequence} is only sufficient for our analytic arguments to work in dimension four. It would be very interesting to know if our approach could be improved so as to work in all dimensions.

\subsection{Acknowledgements}

JF would like to thank Dmitri Panov for introducing him to Gromov and Thurston's work and for many  discussions on it and related topics. BP would like to thank Richard Griffon for sharing important insights concerning \S2. JF was supported by ERC consolidator grant 646649 ``SymplecticEinstein''. JF and BP were both supported by the FNRS grant MIS~F.4522.15. BP was also the recipient of an FNRS \emph{charg\'e de recherche} fellowship whilst this article was being written. Part of this research was carried out whilst JF was a visitor at MSRI and he thanks both them and the NSF (grant number DMS-1440140).

Finally, the authors would like to thank the anonymous referees for their careful reading of the paper and for their comments which greatly helped improve the exposition.

\section{A spin version of Gromov--Thurston manifolds}\label{spin-GT}

In this section we prove Proposition~\ref{hyperbolic-sequence}. In fact we will find infinitely many sequences $(M_k)$ of compact hyperbolic $n$-manifolds satisfying the conclusions of Proposition~\ref{hyperbolic-sequence}. As is mentioned in the introduction, the sequence $(X_k)$ is then found by taking the $l$-fold cover $X_k \to M_k$ branched along $\Sigma_k$ (for some fixed choice of $l$).

The original construction of Gromov and Thurston uses arithmetic hyperbolic geometry to produce $(M_k)$ satisfying the first two properties of Proposition~\ref{hyperbolic-sequence}. We review this in the next subsection. To bound the volume of $\Sigma_k$ we will make use of recent work of Murillo \cite{murillo}, which does not apply to all the sequences arising from Gromov--Thurston's original construction, but instead to a special subclass of them. Put loosely, we need the manifolds $M_k$ to be spin. We describe how to arrange this in section~\ref{volume-bound-via-murillo}. Section~\ref{proof-of-hyperbolic-sequence} then gives the proof of Proposition~\ref{hyperbolic-sequence}.

\subsection{Gromov and Thurston's construction}
\label{GT-original-construction}

Following Gromov--Thurston \cite{gromov-thurston}, consider the following quadratic form on $\R^{n+1}$:
\[
f(x_0, \ldots, x_n) = -\sqrt{2} x_0^2 + x_1^2 + \cdots + x_n^2
\]
The corresponding pseudo-Riemannian metric on $\R^{n+1}$ restricts to a genuine Riemannian metric on the hyperboloid $H = \{ x : f(x) = - 1,\ x_0 >0\}$. This makes $H$ isometric to hyperbolic space $\H^n$ and gives an identification between the group $\SO(\R^{n+1},f)$ of orientation-preserving isometries of $(\R^{n+1},f)$ which are isotopic to the identity, and the group of orientation preserving isometries of~$\H^n$. 

We write $\Gamma$ for those automorphisms of $f$ which are defined over the ring of integers $\Z[\sqrt{2}]$. Explicitly, 
\[
\Gamma =  \SO(\R^{n+1},f) \cap \GL\left(\Z[\sqrt{2}],n+1\right) 
\]
It is important that the action of $\Gamma$ on $\H^n$ is discrete and cocompact. (This is a classical result in the study of arithmetic hyperbolic manifolds; the use of $\sqrt{2}$ is precisely to ensure the action is cocompact.) The quotient $\H^n/\Gamma$ is a compact hyperbolic orbifold with singularities corresponding to the fixed points of $\Gamma$. 

We now explain how to pass to finite covers to remove the orbifold singularities and increase the injectivity radius. This is a standard technique in hyperbolic geometry, but since the proof is short and simple, and the conclusion so important to the rest of the article, we give the details for the benefit of non-experts. 

Let $\Gamma' \leq \Gamma$ be a finite-index subgroup. Write $i(\Gamma')$ for one-half of the minimal displacement of elements of $\Gamma'$:
\[
i(\Gamma') = \frac{1}{2}\min \{ d(\gamma x, x): x \in \H^n,\ \gamma \in \Gamma',\ \gamma \neq 1\}
\]
The point is that when $i(\Gamma')>0$, the action of $\Gamma'$ on $\H^n$ is fixed-point free and so the quotient $\H^n/\Gamma'$ is \emph{smooth} (and still compact, because $\Gamma'$ has finite index in $\Gamma$). In this case, the injectivity radius of $\H^n/\Gamma'$ is one-half the length of the shortest closed geodesic, which is exactly this quantity $i(\Gamma')$.

We will find a sequence $(\Gamma_k)$ of finite-index subgroups for which $i(\Gamma_k) \to \infty$. The key is the following simple lemma. Let $d \geq0$ and set
\[
\hat{F}_d = \{ \gamma \in \Gamma, \gamma \neq 1 : \exists\, x \in \H^n \text{ with } d(\gamma x, x) \leq d\}
\]
$\Gamma$ acts on $\hat{F}_d$ by conjugation and we write $F_d$ for the quotient.

\begin{lemma}
For any $d \geq0$, the set $F_d$ is finite
\end{lemma}
\begin{proof}
Suppose for a contradiction that $F_d$ is not finite. Let $[\gamma_i] \in F_d$ be an infinite sequence  and $x_i \in \H^n$ such that $d(\gamma_i x_i, x_i) \leq d$. Since $\Gamma$ is cocompact, by acting by conjugation we can assume that the $x_i$ all lie in a compact set.  So a subsequence, which we continue to denote by $x_i$, converges to a point $x$. Then by the triangle inequality and for sufficiently large $i$, 
\[
d(\gamma_i x, x) \leq d(\gamma_i x, \gamma_i x_i) + d(\gamma_i x_i, x_i) + d(x_i,x)
=
d(\gamma_i x_i,x_i) + 2 d(x_i,x)
\leq 
d + 1
\]
This means that the points $\gamma_i x$ lie in a compact set and so either the stabiliser of $x$ is infinite or the orbit of $x$ has an accumulation point, contradicting that the action of $\Gamma$ is discrete.
\end{proof}

Given $k \in \N$, write $\Gamma_k \leq \Gamma$ for the kernel of the homomorphism given by reducing the entries $\mod k$, i.e., the homomorphism
\[
\Gamma \hookrightarrow 
\GL(\Z[\sqrt{2}],n+1) 
\to 
\GL\left((\Z/k\Z)[\sqrt{2}],n+1\right)
\]
induced by $\Z \to \Z/k\Z$. The kernel $\Gamma_k$ is a finite-index normal subgroup. Fix $d >0$. Since $F_d$ is finite, for any $k$ sufficiently large, $\Gamma_k$ contains none of the corresponding conjugacy classes. It follows that $N_k = \H^n/\Gamma_k$ is a compact hyperbolic manifold with injectivity radius at least $d$. This gives a sequence of compact hyperbolic manifolds $(N_k)$ with $i(N_k) \to \infty$. 

The next step is to find the totally geodesic codimension~2 submanifolds. To do this, we may need to pass to a finite cover $M_k \to N_k$. The argument goes as follows. Denote by $s$ the reflection of $\H^{n}$ in the $x_1$ coordinate and by $r$ the rotation of $\H^n$ by $\pi/2$ in the $(x_1,x_2)$-plane. Explicitly,
\begin{align*}
s(x_0,\ldots ,x_n) & = (x_0, -x_1, ,x_2, x_3, \ldots , x_n)\\
r(x_0,\ldots ,x_n) &= (x_0, -x_2, x_1, x_3, \ldots ,x_n)
\end{align*}
It is simple to check that for any $k \in \N$, we have $s\Gamma_k s^{-1} = \Gamma_k$ and $r \Gamma_k r^{-1} = \Gamma_k$. It follows that $s$ and $r$ descend to the compact manifolds $N_k$ where they generate an isometric action of the dihedral group $D_4$. (For generalisations of the construction to include higher order dihedral groups, as well as further properties of Gromov--Thurston manifolds, we refer to \cite{kapovich}.) 

Write $H_k \subset N_k$ for the fixed locus of $s \colon N_k \to N_k$. It is a totally geodesic hypersurface. Note that $H_k$ may be disconnected and may have non-orientable components. We will first pass to a cover in which the preimage of $H_k$ is separating. 

\begin{definition}
  We say that a hypersurface $H \subset N$ \emph{separates} a  manifold $N$ if its complement can be written as the disjoint union of two non-empty open sets $N \setminus H = U \cup V$ with $H = \overline{U} \cap \overline{V}$. Notice that a separating hypersurface is necessarily co-orientable and so, when $N$ is also orientable, $H$ is necessarily orientable too. 
\end{definition}

\begin{lemma}\label{separation}
If a hypersurface $H\subset N$ does not separate $N$ then there is a non-trival double cover $M \to N$ in which the preimage of $H$ separates.   
\end{lemma}

\begin{proof}
$H$ determines a class $[H] \in H_{n-1}(N, \Z_2)$ and hence a homomorphism
\[
\rho \colon \pi_1(M) \to  \Z_2
\]
factoring through $H_1(N,\Z_2)$, which counts the parity of the number of intersections of a generic loop with $H$. 

We first claim that $H$ separates if and only if $\rho$ is trivial. In one direction, if $H$ separates and $N \setminus H = U \cup V$ then by counting the number of times a loop enters $U$ say, one sees that every loop must cross $H$ an even number of times and so $\rho$ is trivial. Conversely, if $\rho$ is trivial, we can define a decomposition of $N\setminus H$ as follows. Pick a base-point $p \in N \setminus H$. Given $x \in N \setminus H$, join it to $p$ by a path $l$ which meets $H$ transversely. Define $f(p) \in \Z_2$ to be the parity of intersection of $l$ and $H$. Since $\rho$ is assumed to be trivial, the definition of $f$ doesn't depend on the choice of $l$; the result is a locally constant function and putting $U = f^{-1}(0)$, and $V = f^{-1}(1)$ gives the required decomposition.

So if $H$ doesn't separate, $\ker \rho$ is an index~2 normal subgroup. Let $\pi \colon M \to N$ denote the corresponding double cover and $H'$ the preimage of $H$. We now repeat the above discussion for $H' \subset M$. The resulting homomorphism $\rho' \colon \pi_1(M) \to \Z_2$ is given by $\rho \circ \pi_*$. Since $\im \pi_* = \ker \rho$ it follows that $\rho'$ is trivial and so $H'$ separates as claimed.
\end{proof}

We remark in passing that when $H$ does not separate, the new separating hypersurface $H'$ definitely has different topology. Each oriented connected component of $H$ is replaced by two identical components in $H'$, whilst each non-oriented connected component of $H$ is replaced by its oriented double cover in $H'$. (One can see this by considering what happens to each component in a tubular neighbourhood.)

Returning to our discussion of Gromov and Thurston's construction, we will pass to a finte cover of $M_k \to N_k$ in which the preimage of $H_k$ separates and, moreover, to which the rotation $r$ lifts.

\begin{lemma}\label{good-covers}
For all $k$, there is a finite cover $\pi \colon M_k \to N_k$ for which the preimage $H'_k = \pi^{-1}(H_k)$ separates and to which the rotation $r$ lifts to a map $r \colon M_k \to M_k$.  
\end{lemma}
\begin{proof}
Let $\rho_k \colon \pi_1(N_k) \to \Z_2$ denote the homomorphism in the proof of Lemma~\ref{separation}, which detects whether or not $H_k$ separates $N_k$. We must find a finite index subgroup of $\pi_1(N_k)$ which is both contained in $\ker \rho_k$ and is invariant under $r_*$ (the map induced on $\pi_1(N_k)$ by the map $r$). One checks that this holds for the following subgroup:
\[
\Gamma'_k = \ker \rho_k \cap r_*(\ker \rho_k) \cap r_*^2 (\ker \rho_k) \cap r_*^3(\ker \rho_k)
\qedhere
\]
\end{proof}

We can now describe the null-homologous totally geodesic codimension~2 submanifold $\Sigma_k \subset M_k$. Consider the hypersurface $H'_k$ and its image $H''_k :=r(H'_k)$ under the rotation. Put $S_k = H'_k \cap H''_k$. To check that $H'_k$ and $H''_k$ meet transversely it suffices to check the same is true for the preimages  in $\H^{n}$ of each pair of intersecting components; this leads to a pair of distinct hyperplanes, meeting in a totally geodesic copy of $\H^{n-2}$, which projects down to the corresponding component of $S_k$. We know that $H'_k$ separates $M_k$, it follows that $H''_k$ is also separating. In particular, both are orientable and hence $\Sigma_k$ is too. To see that $\Sigma_k$ is null-homologous, fix a decomposition $M_k \setminus H'_k = U \cup V$ of $M_k$ and let $Z_k = H''_k \cap U$. Then $Z_k$ is an $(n-1)$-chain with boundary $\del Z_k = \Sigma_k$.

Finally, we check the condition on the normal injectivity radius of $\Sigma_k$. Since $M_k$ is hyperbolic there are no conjugate points on geodesics. This means there is a path $\gamma$  in $M_k$ starting and ending on $S_k$, which it meets at right-angles and with $L(\gamma) = 2i(S_k,M_k)$. Acting on $\gamma$ by $r^2$ we find another such path which when joined to $\gamma$ gives a smooth geodesic loop of length $2L(\gamma)$. But the length of any geodesic loop is at least $2i(M_k)$ and so $i(S_k,M_k) = \frac{1}{2}L(\gamma) \geq \frac{1}{2}i(M_k)$ as claimed.

\subsection{Murillo's volume estimate}
\label{volume-bound-via-murillo}

So far, we have followed Gromov and Thurston's construction precisely. We now turn to the volume bound, for which we need to modify the construction slightly.

The important new ingredient is a bound of the form
\[
\vol(M_k) \leq Ae^{Ci(M_k)}
\]
for constants $A$ and $C$ which are independent of $k$. It is not difficult to show that for the sequence $M_k$ constructed by Gromov and Thurston, such constants $A(n), C(n)$ exist, depending only on the dimension $n$. There is a short self-contained proof of this fact in \S4.1 of~\cite{guth-lubotzky}. However, for our purposes it is important to know the precise value of $C(n)$. 

For the congruence coverings discussed here, Katz--Schaps--Vishne  \cite{katz-schaps-vishne} proved that $C(2)= 3/2$ and $C(3) = 3$. (See also the related work \cite{buser-sarnack} of Buser--Sarnack.) For us, of course, the interest is in $n \geq 4$. This was treated in a recent article by Murillo \cite{murillo}, which shows that, with a couple of caveats, the optimal inequality has 
\[
C(n) = \frac{n(n+1)}{4}
\]
The first caveat is that Murillo's argument works for sequences of congruence subgroups starting with $\Spin(\R^{n+1},f)$ rather than $\SO(\R^{n+1},f)$. The second is that the congruences must be determined by \emph{prime} ideals $I \leq \Z[\sqrt{2}]$ and not just reduction modulo an arbitrary integer.

We give a brief description of Murillo's theorem, following the original article closely (where the reader can also find detailed justifications for everything in this section). In order to keep the arithmetic to a minimum, we continue to work with the quadratic form $f$ which is defined over $\Q(\sqrt{2})$ but in fact everything holds much more generally, for admissible quadratic forms defined over totally real number fields. The interested reader can find the details neatly summarised in Murillo's article. 

We first give the analogue of $\Gamma \subset \SO(\R^{n+1},f)$, namely the subgroup $\overline{\Gamma} \leq \Spin(\R^{n+1},f)$ of elements defined over the ring of integers $\Z[\sqrt{2}]$. To do so, we use the Clifford representation of $\Spin(\R^{n+1},f)$. The group $\Spin(\R^{n+1},f)$ is a subgroup of the  Clifford algebra $\Cliff(\R^{n+1},f)$ and so acts on it by left multiplication. To use this to get a matrix representation of $\Spin(\R^{n+1},f)$ we fix a basis of $\Cliff(\R^{n+1},f)$. Choose first a basis $e_0, \ldots , e_n$ of $\R^{n+1}$ with respect to which the innerproduct $g_f$ defined by $f$ is standard:
\[
g_f(e_i,e_j) 
	= 
		\begin{cases}
		-1 &\text{ if }i=j=0\\
		\phantom{-}1 &\text{ if }i=j >0\\
		\phantom{-}0 &\text{ if }i \neq j
		\end{cases}
\]
Then $\Cliff (\R^{n+1},f)$ has as basis the $2^{n+1}$ elements of the form $e_{i_1} \cdots e_{i_r}$ where $i_1 < \cdots < i_r$ and $r = 0, \ldots , n+1$. With respect to this basis, left multiplication by $\Spin$ on $\Cliff$ gives a faithful representation
\[
\rho_L 
\colon 
\Spin\left(\R^{n+1},f\right) \hookrightarrow \GL\left(\R, 2^{n+1}\right)
\]
We now set $\overline{\Gamma}\subset \Spin(\R^{n+1},f)$ to be
\[
\overline{\Gamma} = \im \rho_L \cap \GL\left(\Z[\sqrt{2}],2^{n+1}\right)
\]
$\overline{\Gamma}$ has an explicit description in terms of the above basis for $\Cliff(\R^{n+1},f)$. Given $J = (i_1, \ldots, i_r)$ we write $e_J = e_{i_1} \cdots e_{i_r}$ for the corresponding basis element of the Clifford algebra. Then
\[
\overline{\Gamma}
=
\left\{ 
\gamma = \sum_{|J|\text{ even}} c_J e_J \in \Spin(\R^{n+1},f) : c_J \in \Z[\sqrt{2}] \text{ for }J\neq \emptyset,\ c_\emptyset - 1 \in  \Z[\sqrt{2}] 
\right\}
\]
$\overline{\Gamma}$ acts on $\H^n$ via 
\[
\overline{\Gamma} \hookrightarrow \Spin(\R^{n+1},f) \to \SO(\R^{n+1},f)
\]
(where the second arrow is the standard double cover). The crucial fact is that \emph{the resulting action of $\overline{\Gamma}$ is discrete and cocompact} and so $\H^n/\overline{\Gamma}$ is a compact hyperbolic orbifold. (Again this is a foundational fact in the study of arithmetic hyperbolic manifolds.)

We now pass to finite covers. Let $I \subset \Z[\sqrt{2}]$ denote an ideal. We obtain a normal subgroup $\Gamma_I \leq \overline{\Gamma}$ as the kernel of the homomorphism
\[
\overline{\Gamma} \to \GL \left( \Z[\sqrt{2}]/I, 2^{n+1}\right)
\]
Explicitly,
\begin{equation}
\Gamma_I
=
\left\{ 
\gamma = \sum_{|J|\text{ even}} c_J e_J \in \overline{\Gamma} : 
c_J \in I \text{ for }J\neq \emptyset,\ c_\emptyset - 1 \in I
\right\}
\label{explicit-form-Gamma-I}
\end{equation}
We are now in a position to state Murillo's Theorem.

\begin{theorem}[Murillo's volume bound \cite{murillo}]\label{murillo}
Let $I_k \subset \Z[\sqrt{2}]$ be a sequence of prime ideals with $|\Z[\sqrt{2}]/I_k| \to \infty$ and write $\Gamma_{I_k} \leq \overline{\Gamma}$ for the corresponding normal subgroups of $\overline{\Gamma}$. Then for sufficiently large $k$, the quotient $M_k = \H^n/\Gamma_{I_k}$ is smooth and there is a constant $A$ such that for all $k$, 
\[
\vol(M_k) \leq A \exp\left( \frac{n(n+1)}{4} i(M_k) \right)
\]
\end{theorem}

The key observation in Murillo's proof is to control the hyperbolic displacement of an element $\sigma \in \Spin(\R^{n+1},f)$ by the size of the coefficient of $e_{\emptyset}$ in the expression $\sigma = \sum c_Je_J$ of $\sigma$ in terms of the chosen basis of $\Cliff(\R^{n+1},f)$. From here he is able to control the minimal displacement $i(\Gamma_I)$ from below in terms of the cardinality of the quotient $\Z[\sqrt{2}]/I$. At the same time, the index $[\overline{\Gamma}: \Gamma_I]$ can be controlled from above in terms of this same quantity and, since volume is proportional to index, this leads to Theorem~\ref{murillo}.

\subsection{Proof of Proposition~\ref{hyperbolic-sequence}}
\label{proof-of-hyperbolic-sequence}

We now give the proof of Proposition~\ref{hyperbolic-sequence}. Let $I_k \subset \Z[\sqrt{2}]$ be a sequence of prime ideals as in Murillo's Theorem, let $\Gamma_{I_k} \subset \overline{\Gamma}$ and write $M_k = \H^{n}/\Gamma_{I_k}$ for the corresponding hyperbolic manifolds. We can for instance take $I_k = p_k \Z[\sqrt{2}]$ for a suitable increasing sequence of prime numbers $(p_k)$. Just as before, $i(M_k) \to \infty$ as $k \to \infty$. 

To find the nullhomologous totally geodesic codimension~2 submanifold $\Sigma_k \subset M_k$ we copy the same argument. Recall that the natural action  $\rho \colon \Spin(\R^{n+1},f) \to \SO(\R^{n+1},f)$ is given by $\rho(\sigma)(v) = \sigma v\sigma^{-1}$ where we treat $v \in \R^{n+1}$ as an element of the Clifford algebra, and the product on the righthand side of this formula is the Clifford product. We can also represent reflections in a similar way. Let $e_0, \ldots, e_n$ be a basis of $\R^{n+1}$ for which $f$ is standard and consider the linear transformation $\hat{s}$ of $\Cliff(\R^{n+1},f)$ given by $\hat{s} (c) = e_1 c e_1$. Note that for any multi-index $J$,  
\[
\hat{s}(e_J) 
	= 
		\begin{cases}
		-1^{|J|+1}e_J &\text{ if }1 \notin J\\
		-1^{|J|}e_J &\text{ if } 1 \in J
		\end{cases}
\]
In particular, $\hat{s}$ preserves $\R^{n+1} \subset \Cliff(\R^{n+1},f)$ where it acts as reflection in the hyperplane orthogonal to~$e_1$. Moreover, from the description \eqref{explicit-form-Gamma-I}  of $\Gamma_{I_k}$, it follows that $\hat{s}(\Gamma_{I_k}) = \Gamma_{I_k}$. So $\hat{s}$ descends to an isometry $s$ of $M_k$. Similarly, there is a second reflection, this time in the hyperplane orthogonal to $e_1+e_2$, coming from the linear map of $\Cliff(\R^{n+1},f)$ given by
\[
\hat{s}' \colon c \mapsto \frac{1}{2} (e_1+e_2)c(e_1+e_2)
\]
Again, $\hat{s}'(\Gamma_{I_k}) = \Gamma_{I_k}$ and so this descends to an isometry $s'$ of $M_k$. The reflections $s,s'$ generate an action of the dihedral group $D_4$; in particular we have a reflection $s$ and a rotation $r = ss'$ and we are in precisely the same situation discussed in Section~\ref{GT-original-construction}. Write $H_k$ for the fixed locus of $s$; as before we can pass to a finite cover to which the rotation $r$ lifts and in which the preimage of $H_k$ separates. We abuse notation by also denoting this cover $M_k$. Just as in the original Gromov--Thurston construction, the intersection $H_k 
\cap r(H_k) = \Sigma_k$ is the totally geodesic null-homologous submanifold we are looking for and an identical argument as before shows that the normal injectivity radius satisfies $i(\Sigma_k,M_k) \geq \frac{1}{2} i(M_k)$. 

It remains to prove the volume bound~\eqref{vol-bound-we-want} on $\Sigma_k$. We control the volume of $\Sigma_k$ in two steps. Write $i(H_k,M_k)$ for the normal injectivity radius of $H_k \subset M_k$. By considering the volume of an embedded tubular neighbourhood of $H_k$ of maximal radius we have
\begin{equation}
\vol(H_k) e^{(n-1)i(H_k,M_k)} \leq A_1 \vol(M_k)
\label{vol-H1}
\end{equation}
where $A_1$ is independent of $k$. Similarly, by considering an embedded tubular neighbourhood of $\Sigma_k$ in $H_k$ of maximal radius, we find a constant $A_2$ such that
\begin{equation}
\vol(\Sigma_k) e^{(n-2)i(\Sigma_k,H_k)} \leq A_2 \vol(H_k)
\label{vol-Sigma}
\end{equation}
Now $i(\Sigma_k,H_k) \geq i(\Sigma_k, M_k) \geq \frac{1}{2}i(M_k)$ and, similarly, $i(H_k,M_k) \geq \frac{1}{2}i(M_k)$. Using this and putting \eqref{vol-H1} and \eqref{vol-Sigma} together we see that
\begin{equation}
\vol(\Sigma_k) 
	\leq 
		A_1 A_2 \vol (M_k) 
		e^{ \frac{3 -2n}{2} i (M_k)}
\label{ready-for-murillo}
\end{equation}
The bound~\ref{vol-bound-we-want} now follows from~\eqref{ready-for-murillo} and Theorem~\ref{murillo}.

\section{The approximate solution}\label{approximate-solution}

In this section we give the construction of the approximate solutions to Einstein's equations. Recall that in the previous section we constructed a sequence of hyperbolic $n$-manifolds $(M_k,h_k)$, each containing a totally geodesic hypersurface $H_k \subset M_k$ whose boundary $\Sigma_k$ is also totally geodesic. The injectivity radius $i(M_k)$ of $M_k$ tends to infinity with $k$ and $\frac{1}{2}i(M_k)$ is a lower bound for the normal injectivity radius of $\Sigma_k \subset M_k$.  We denote by $p \colon X_k \to M_k$ the cyclic $l$-fold cover branched along $\Sigma_k$. We abuse notation by using $\Sigma_k$ to also denote the branch locus in $X_k$.  

Define a function $r \colon X_k \to \R$ by setting $r(x)$ to be the distance of $p(x) \in M_k$ from the branch locus $\Sigma_k$. 
As a notational convenience, we set $u = \cosh(r)$. Write $U_{k,\max} = \cosh(\frac{1}{2}i(M_k))$ and pick a sequence $(U_k)$ which tends to infinity, with $U_k < \frac{1}{2} U_{k,\max}$. The main result of this section is the following. 

\begin{proposition}\label{prop-approx-solution}
For each $k$, there is a smooth Riemannian metric $g_k$ on $X_k$ with the following properties:
\begin{enumerate}
\item \label{approx-Einstein}
For any $m \in \N$ and $0 \leq \eta < 1$, there is a constant $A$ such that for all $k$,
\[
\| \Ric(g_k) + (n-1)g_k\|_{C^{m,\eta}} \leq A U_k^{1-n}
\]
\item \label{negatively-curved}
 There is a constant $c>0$ such that for all $k$, $\sec(g_k) \leq - c$.
\item \label{support}
$\Ric(g_k) + (n-1)g_k$ is supported in the region $\frac{1}{2}U_k < u < U_k$. 
\item \label{curvature-bounded}
For any $m \in \N$, there exists a constant $C$ such that for all $k$, $\| \Rm(g_k) \|_{C^m}\leq C$.
\end{enumerate}
\end{proposition}

The metric $g_k$ will be given by interpolating in the region $\frac{1}{2}U_k< u < U_k$ between a model Einstein metric defined on a tubular neighbourhood of $\Sigma_k \subset X_k$ and the hyperbolic metric $p^*h_k$ pulled back via the branched cover $p \colon X_k \to (M_k,h_k)$ on the complement of this tubular neighbourhood. In Proposition~\ref{prop-approx-solution} the H\"older norms are defined with respect to the metric $g_k$ (see definition \ref{Holder-definition} for the explicit definition used in this paper for the H\"older norms). We begin by describing the model.

\subsection{The model Einstein metric}

Write $(\H^n, h)$ for hyperbolic space of dimension $n$. Denote by $S \subset \H^n$ a totally geodesic copy of $\H^{n-2}$. We can write $h$ as
\[
h = \diff r^2 + \sinh^2(r)\diff \theta^2 + \cosh^2(r) h_S
\]
where $h_S$ is the hyperbolic metric on $S$. Here, $(r,\theta) \in (0,\infty) \times S^1$ are polar coordinates on the totally geodesic copies of $\H^2$ which are orthogonal to $S$. The hypersurfaces given by setting $\theta$ constant are the totally geodesic copies of $\H^{n-1}$ containing $S$. In fact, it will be more convenient to use the coordinate $u = \cosh(r)$; the hyperbolic metric then becomes
\begin{equation}
\label{hyperbolic-metric-u}
h = \frac{\diff u^2}{u^2 -1} + (u^2 -1) \diff \theta^2 + u^2 h_S
\end{equation}
This expression is valid for $(u, \theta) \in (1, \infty) \times S^1$.

We will consider a family $g_a$ of Einstein metrics depending on a parameter $a \in\R$.  When $a=0$, we recover $h$, whilst for $a \neq 0$ the metric has a cone singularity along $S$, with cone angle varying with $a$. By an appropriate choice of $a$, the metric will have the correct cone angle to become smooth when pulled back by an $l$-fold cover ramified along $S$.

The metrics we will consider all have the form
\begin{equation}
\label{model-ansatz}
g = \frac{\diff u^2}{V(u)} + V(u) \diff \theta^2 + u^2 h_S
\end{equation}
where $V$ is a smooth positive function.

\begin{proposition}\label{model-is-Einstein}
The metric \eqref{model-ansatz} solves $\Ric(g) = -(n-1)g$ precisely when
\begin{equation}\label{V-definition}
V(u) = u^2 - 1 + \frac{a}{u^{n-3}}
\end{equation}
for some $a\in \R$.
\end{proposition}

\begin{proof}
The proof of this is a direct and standard calculation, but we give the details since certain parts (notably the description of the Levi-Civita connection and curvatures) will be useful later. We use the convention that our indices $i,j$ run between $1, \ldots , n-2$. Let $f^i$ be an orthonormal coframe for $(S,h_S)$, and write $\omega^{i}_j$ for the connection matrix of the Levi-Civita connection of $h_S$, i.e, $\nabla^{h_S} f^i = \omega^i_j\otimes f^j$. Let $W^2 =V$, then
\[
e^i = uf^i,\quad
e^{n-1} = W^{-1}\diff u,\quad
e^{n} = W \diff \theta
\]
is an orthonormal coframe for $g$. We have the following formulae for the exterior derivatives:
\begin{align*}
\diff e^i
	&=
		\omega^i_j\wedge e_j - \frac{W}{u} e^i \wedge e^{n-1}\\
\diff e^{n-1}
	&=
		0\\
\diff e^n
	&=
		W' e^{n-1} \wedge e^n		
\end{align*}
From here we can write down the connection matrix, $A$, for the Levi-Civita connection of~$g$. To ease notation, we write $\mathbf{e}$ for the column vector with entries $e_1, \ldots, e_{n-2}$.
\begin{equation}
\label{LC-model}
\left(\begin{array}{l}
\nabla \mathbf{e}\\
\nabla e^{n-1}\\
\nabla e^n
\end{array}
\right)
=
\left(\begin{array}{ccc}
\omega& - \frac{W}{u}\mathbf{e} & 0\\
\frac{W}{u}\mathbf{e}^t & 0 & W' e^n\\
0 & - W' e^n & 0
\end{array}
\right)
\otimes
\left(\begin{array}{l}
\mathbf{e}\\
e^{n-1}\\
e^n
\end{array}
\right)
\end{equation}
Next we compute the curvature matrix, $\Omega = \diff A - A \wedge A$:
\begin{align*}
\Omega_{ij} 
	&=
		(\diff \omega - \omega\wedge \omega)_{ij}
		+ \frac{W^2}{u^2} e^i\wedge e^j \\
\Omega_{i, n-1}
	&=
		\frac{W'W}{u} e^i \wedge e^{n-1}\\
\Omega_{i, n}
	&=
		\frac{W'W}{u} e^i \wedge e^n \\
\Omega_{n-1, n}
	&=
		\left( W''W+ (W')^2  \right)e^{n-1} \wedge e^n 
\end{align*}
From here we can read off the components of the curvature tensor, $R_{abcd} = - \left\langle \Omega_{ab}, e^c \wedge e^d \right\rangle$. We use the convention that a Roman index takes the values $1 ,\ldots, n-2$, whilst a Greek index takes the values $n-1,n$. Since $(S,h_S)$ is hyperbolic, the non-zero components are:
\begin{align*}
R_{ijkl} &= -\frac{1+W^2}{u^2}(\delta_{ik}\delta_{jl}-\delta_{il}\delta_{jk})\\
R_{i\mu i \mu} & = - \frac{W'W}{u}\\
R_{\mu \nu \mu \nu} &= -\left( W''W + (W')^2 \right)
\end{align*}
whilst the remaining components are zero:
\begin{align*}
R_{i\mu\nu\rho} & = 0\\
R_{ij\mu\nu}& = 0\\
R_{i\mu j \nu} & = 0 \text{ unless }i=j \text{ and }\mu=\nu\\
R_{ijk\mu} & = 0\\
R_{\mu\nu\rho\sigma} &= 0 \text{ unless }\mu=\rho \text{ and }\nu=\sigma
\end{align*}
(One can also use symmetry arguments to show these components vanish. For example, for any point $p$ there is an isometry  given by an inversion in $S$ which fixes $p$, reverses the $e^i$ and leaves the $e^\mu$ unchanged. It follows that any component with an odd number of Roman indices must vanish.)

From here we can compute the Ricci curvature, which is diagonal:
\begin{align}
R_{ij}
	&=
		\left[
		\left(3-n \right) \frac{1+W^2}{u^2} - \frac{2W'W}{u}
		\right] \delta_{ij}
		\label{model-Ric-ii}\\
R_{\mu\nu}
	&=
		\left[
		\left( 2- n \right)\frac{W'W}{u} - \left( W''W+(W')^2 \right)
		\right]\delta_{\mu\nu}
		\label{model-Ric-nn}\\
R_{i\mu}
	&= 
		0\label{model-Ric-in}
\end{align}
To find an Einstein metric with $\Ric = -(n-1)g$, it suffices to solve
\begin{equation}\label{Einstein-condition}
\left(3-n \right) \left( 1+W^2 \right) - 2W'Wu = -(n-1)u^2
\end{equation}
because this sets $R_{ii} = -(n-1)$ and then differentiating \eqref{Einstein-condition} with respect to $u$ we also get that $R_{\mu\mu} = -(n-1)$. Equation \eqref{Einstein-condition} rearranges to give
\[
V' + \frac{n-3}{u}V = (n-1) u - (n-3) u^{-1}
\]
and a simple integrating factor shows that the solutions of this equation are
\[
V(u) = u^2 - 1 + \frac{a}{u^{n-3}}
\]
(where $a \in \R$) as claimed.
\end{proof}

When $V$ is given by \eqref{V-definition}, we denote the metric \eqref{model-ansatz} by $g_a$. We next consider the singularity of the metric $g_a$. The metric is smooth for those values of $u$ for which $0 < V(u) < \infty$. Write $u_a$ for the largest root of $V$. At least when $u_a >0$, the metric $g_a$ is defined for $u \in (u_a, \infty)$. The metric $g_a$ has a cone singularity at $u=u_a$. The next lemma describes how the cone angle depends on $a$.
\begin{lemma}\label{cone-angle}
Let
\begin{equation}\label{v-and-amax}
v = \sqrt{\frac{n-3}{n-1}},\quad \quad
a_{\max} = \frac{2}{n-1} v^{n-3}
\end{equation}
\begin{enumerate}
\item We have $u_a>0$ if and only if $a \in (-\infty, a_{\max}]$. The map $a \mapsto u_a$ is a decreasing homeomorphism $(-\infty, a_{\max}] \to [v, \infty)$.
\item When $a \in (-\infty, a_{\max})$, the metric $g_a$ has a cone singularity along $S$ at $u = u_a$, with cone angle $2\pi c_a \in (0,\infty)$. 
\item The map $a \mapsto c_a$ is a decreasing homeomorphism $(-\infty, a_{\max}] \to [0, \infty)$. In particular, as $a$ runs from $0$ to $a_{\max}$, the cone angle takes every value from $2\pi$ to $0$ precisely once.
\end{enumerate}
\end{lemma}

\begin{proof}
The number $u_a$ is the largest solution of $f(u) = a$ where $f(u) = (1-u^2)u^{n-3}$, so the first assertion follows from a sketch of the graph of $f$. We now turn to the second assertion. For $u$ near to $u_a$ we have
\[
V(u) = 2c_a (u-u_a) +O\left( (u-u_a)^2 \right)
\]
where $c_a = \frac{1}{2}V'(u_a)$. The substitution $s = \sqrt{\frac{2}{c_a} (u -u_a) }$ shows that for $s \sim 0$, the metric becomes
\[
g_a \sim \diff s^2 + c_a^2 s^2 \diff \theta^2 + (w  +u_a)^2 h_V
\]
Therefore the metric has a conical singularity at $s=0$, i.e., at $u = u_a$, with a cone angle of $2\pi c_a$. Finally,  $V'(u) = 2u + (3-n)a u^{2-n}$ so
\begin{align*}
2c_a	
	&=
		2u_a + (3-n)au_a^{2-n}\\
	&=
		(n-1)u_a + (3-n)u_a^{-1}	
\end{align*}
where in the second line we have used that $f(u_a)=0$. By the first assertion we deduce that $a \mapsto c_a $ is a homeomorphism $(-\infty, a_{\max}] \to [0, \infty)$. Finally, $c_0 =1$ and so as $a$ runs from $0$ to $a_{\max}$ the cone angle varies from $2\pi$ to $0$ exactly as claimed.
\end{proof}

We now prove that the sectional curvatures of $g_k$ are all negative:

\begin{lemma}\label{model-negatively-curved}
For $a \in (0,a_{\max})$, the metric $g_a$ is negatively curved, with all sectional curvatures satisfying
\[
\mathrm{sec} \leq -1 + \frac{n-3}{2}a u_a^{1-n} <0
\]
\end{lemma}

\begin{proof}
In the proof of Proposition~\ref{model-is-Einstein}, we computed the curvature tensor of $g_a$. The only non-zero components are:
\begin{align*}
R_{ijkl} 
	&= 
		- (1 + au^{1-n})\left( \delta_{ik}\delta_{jl} - \delta_{il}\delta_{jk} \right)\\
R_{\mu \nu \rho \sigma}
	&=
		-\left( 1 + \frac{(n-3)(n-2)}{2} a u^{1-n} \right)
			\left( \delta_{\mu \rho}\delta_{\nu\sigma} 
				- \delta_{\mu \sigma}\delta_{\nu\rho} \right)\\
R_{i\mu i\mu}
	&=
		-1 + \frac{n-3}{2}au^{1-n}
\end{align*}
From this, the vanishing of the other components of the curvature, and the fact that $a>0$ it follows that the largest sectional curvatures are
\[
R_{i\mu i\mu} 
	\leq
		-1 + \frac{n-3}{2}au_a^{1-n}
	<
		-1 + \frac{n-3}{2}a_{\max}v^{1-n}
	= 0
\qedhere
\]
\end{proof}

\begin{remark}
Notice that as $u \to \infty$, the metric $g_a$ approaches the hyperbolic metric. The metrics $g_a$ are Riemannian analogues of static generalized Kottler metrics, which are solutions of the Lorentzian Einstein's equations with a negative cosmological constant. {We refer to Anderson \cite{anderson2} for a description of the conformal infinities of these metrics, and to} Chru\'sciel--Simon \cite{chrusciel-simon} for a study of Kottler spacetimes.
\end{remark}

\subsection{The sequence of interpolated metrics}

We now transfer the model metric to our sequence $(M_k)$ of compact hyperbolic $n$-manifolds. We keep the notation introduced at the beginning of Section 3. We let $(U_k)$ be a sequence tending to infinity with $k$ with $U_k < \frac{1}{2}U_{k,\max}$.

Let $\Sigma^0_k$ denote a connected component of $\Sigma_k$. Using geodesics orthogonal to $\Sigma_k^0$, we can set up a tubular neighbourhood of $\Sigma^0_k$ in which the hyperbolic metric on $M_k$ is given by
\[
 \frac{\diff u^2}{u^2-1} + (u^2-1)\diff \theta^2 + u^2 h_{\Sigma}
\]
Here $h_\Sigma$ is the hyperbolic metric on $\Sigma^0_k$ and $u = \cosh(r)$ where $r$ is the distance to $\Sigma^0_k$. The hypersurface $\theta = 0$ corresponds to the totally geodesic hypersurface $H_k$; in general $\theta(p)$ is the angle that the shortest geodesic from $p$ to $\Sigma^0_k$ makes with $H_k$. This expression is valid for $(u, \theta) \in [1,U_{k,\max})\times S^1$. 

Let $a \in (0, a_{\max})$.  We define a new metric near $\Sigma^0_k$ interpolating between $g_a$ and the hyperbolic metric as follows. Let $\chi \colon \R \to [0,\infty)$ be a smooth function with $\chi(u) = 1$ for $u \leq 1/2$ and $\chi(u) = 0$ for $u \geq 1$. Write 
\begin{equation}
V(u) = u^2 - 1 + \frac{a}{u^{n-3}} \chi\left(\frac{u}{U_k}\right)
\label{interpolating-V}
\end{equation}
and consider the corresponding metric 
\[
\frac{\diff u^2}{V} + V \diff \theta^2 + u^2 h_{\Sigma}
\]
The factor $\chi(u/U)$ has the effect of interpolating between the Einstein model of the previous section for $u \leq \frac{1}{2}U_k$ and the hyperbolic metric for $u \geq U_k$.  Since the model is close to hyperbolic at large distances from $\Sigma^0_k$, when $U_k$ is large this interpolation does not change the metric very much. As we will see, this means that the result is close to Einstein.

We also note that in terms of the intrinsic distance $r$ from $\Sigma^0_k$, we are using the Einstein model for $r< \log U_k$ and the hyperbolic metric for $r>\log U_k + \log 2$, so the band on which the interpolation takes place has fixed geodesic width, independent of~$k$.

The expression for the interpolated metric is valid for $(u, \theta) \in (u_a,  U_{k,max}) \times S^1$, where $u_a < 1$ (since $a>0$). We remove the tubular neighbourhood of $\Sigma_k^0$ at a distance $U_k$ and glue this new metric in. The result is a metric on the same manifold, which is smooth across $u=U_k$ and which has a cone singularity along $\Sigma^0_k$ of angle $2\pi c_a$. Carrying out this procedure at every connected component of $\Sigma_k$ we obtain a metric $\tilde{g}_k$ on $M_k$ which is Einstein near $\Sigma_k$, hyperbolic at long distances from $\Sigma_k$, and has cone singularities along each component of $\Sigma_k$, of angle $2\pi c_a$.  

We now pass to the $l$-fold branched cover $p \colon X_k \to M_k$. By Proposition~\ref{cone-angle}, there is a unique value of $a \in (0,a_{\max})$ for which the cone angles of $\tilde{g}_k$ are $2\pi/l$. It follows that the pull-back metric $g_k = p^*\tilde{g}_k$ is smooth on the whole of $X_k$, even across the branch locus. We are now in position to prove Proposition~\ref{prop-approx-solution}.

\begin{proof}[Proof of Proposition~\ref{prop-approx-solution}]
{First, part~\ref{support}  is immediate from the construction, since the only place in which $\Ric(g_k) + (n-1)g_k$ is non-zero is the interpolation region. Let $m \in \N$ and $0 \leq \eta <1$. Using equations \eqref{model-Ric-ii}, \eqref{model-Ric-nn} and \eqref{model-Ric-in} we immediately obtain that  there is a constant $A$, depending on $m$, $\eta$ and $l$, but not on $k$, such that 
\[
\| \Ric(g_k) + (n-1)g_k\|_{C^{m, \eta}} \leq A U_k^{1-n},
\]
which proves part~\ref{approx-Einstein}. }

{We now move to part \ref{negatively-curved}, showing that there is a constant $c > 0$ such that for all large $k$, $\sec(g_k) \leq -c$. For $u \leq \frac{1}{2}U_k$ the metric $g_k$ agrees with the model whose sectional curvatures are uniformly bounded away from $0$, thanks to Lemma~\ref{model-negatively-curved}. For $u \geq U_k$, the metric $g_k$ is hyperbolic, and in the interpolation region $\frac{1}{2}U_k < u < U_k$ we have $\sec(g_k) = -1 + O(U_k^{1-n})$ as shown by Proposition \ref{model-is-Einstein}. Similarly, part~\ref{curvature-bounded} follows from the previous calculations.}
\end{proof}

\section{The inverse function theorem}\label{setting-up-IFT}

Our aim in this and the next section is to show that for all sufficiently large $k$ there is an Einstein metric on $X_k$ near to $g_k$. In this section we will set this up as a question about the inverse function theorem and reduce it to a key analytic estimate. We will then prove this estimate in the case $n=4$ in the following section. 

\subsection{The Bianchi gauge condition}

We will apply the implicit function theorem to a non-linear elliptic map between appropriate Banach spaces. Einstein's equations are diffeomorphism invariant and so not directly elliptic. We deal with this in the standard way (appropriate for Einstein metrics with negative scalar curvature) by adding an additional term, a technique called  \emph{Bianchi gauge fixing}. We describe this briefly here and refer to \cite{anderson,biquard} for proofs of the results we use. 

The following applies to arbitrary closed Riemannian manifolds $(X,g)$, and so we momentarily drop the $k$ subscript to ease the notation. We write $\div_g \colon C^\infty(S^2T^*X) \to C^\infty(T^*X)$ for the divergence of a symmetric 2-tensor. In abstract index notation, we have $(\div_g h)_{p} = -\nabla^qh_{pq}$, where $\nabla$ is the Levi-Civita connection. We write $\div_g^*$ for the $L^2$-adjoint. Again, in index notation, $(\div_g^*\alpha)_{ab} = \nabla_{(a}\alpha_{b)} := \frac12 (\nabla_a \alpha_b + \nabla_b \alpha_a)$. 

A computation gives that the linearisation of the Ricci curvature is
\begin{equation}
(\diff_g \Ric)(s) = \frac{1}{2} \Delta_L(s) - \div^*_g\div_g (s) - \frac{1}{2} \nabla \diff \tr_g(s)
\label{linearised-Ricci}
\end{equation}
where $\Delta_L$ is the Lichnerowicz Laplacian:
\[
\Delta_L (s) = \nabla^*\nabla + \Ric_g \circ s + s \circ \Ric_g - 2\Rm_g(s)
\]
In index notation this is
\[
(\Delta_L s)_{ab}
=
\nabla^p\nabla_p s_{ab}
+
R_{a}^{\phantom{a}p}s_{pb} + R_{b}^{\phantom{b}p}s_{pa}
- 
2s^{pq} R_{apbq}
\]

Define the Bianchi operator $B_g\colon C^\infty(S^2T^*X) \to C^\infty(T^*X)$ by
\begin{equation}
B_g(h) = \div_g h + \frac{1}{2} \diff (\tr_g h)
\label{Bianchi-operator}
\end{equation}
Note that the contracted Bianchi identity gives $B_g(\Ric(g)) =0$. Now given a pair of metrics $g,h$, we write
\begin{equation}
\Phi_g(h) = \Ric(h) + (n-1)h + \div_h^* \left( B_g(h) \right)
\label{Phi}
\end{equation}
Here $\div_h^*$ is the formal adjoint of $\div_h$, taken with respect to the $L^2$ inner product defined by~$h$. We call $\Phi_g$ the \emph{Einstein operator in Bianchi gauge relative to $g$}.

One can check that the addition of this second term produces an elliptic map. We write $L_h$ for the derivative of $\Phi_g$ at $h$. The case $h=g$ is the simplest:
\begin{equation}
L_g(s)
=
\frac{1}{2} \Delta_L s+ (n-1)s
\label{linearised-Einstein}
\end{equation}
The derivative at a general point is slightly more awkward. To describe it, we introduce the following notation. Given a metric $h$, a section $s \in C^\infty(S^2T^*X)$ and a 1-form $\alpha \in C^\infty(T^*X)$ we consider the quantity
\[
 \lim_{t \to 0} \frac{(\div^*_{h+ts} - \div_h^*)\alpha}{t}
\]
i.e., the infinitesimal change of $\div^*_h(\alpha)$ when $h$ moves in the direction $s$. This has an expression of the form $\alpha * \nabla s$ where $\nabla$ is the Levi-Civita connection of $h$ and $*$ denotes some universal algebraic contraction. (See for example, the discussion in \S2.3.1 of \cite{topping}.) We can now give the formula for $L_h$:
\begin{multline}
L_h(s)
	=
		\frac{1}{2} \Delta_{L,h}(s) + (n-1)s
		+
		\div_h^*(\div_g - \div_h)(s)\\
		 +
		 \frac{1}{2}\nabla_h \diff \left( \tr_g(s) - \tr_h(s) \right)
		 +
		 B_g(h) *\nabla_h(s)	
\label{linearised-Einstein-general}
\end{multline}
(where $\nabla_h$ denotes the Levi-Civita connection of $h$). This follows by direct differentiation of \eqref{Phi} together with the linearisation of Ricci curvature \eqref{linearised-Ricci}. Since $B_g(g) =0$, \eqref{linearised-Einstein} follows from \eqref{linearised-Einstein-general}.

The precise expression for $L_h$ is not important in what follows. What is essential is that it is locally Lipschitz continuous in $h$. More precisely:
\begin{lemma}\label{Lipschitz-continuity-L}
Fix an integer $m\geq 2$ and $0\leq \eta < 1$. Given $K>0$ there exist constants $\delta,C >0$ such that if $g, h, \tilde{h}$ are Riemannian metrics on the same $n$-dimensional manifold with 
\begin{align*}
\|g - h\|_{C^{m,\eta}}, \| g -\tilde{h} \|_{C^{m,\eta}} < \delta\\
\| \Rm_g \|_{C^{m-2,\eta}} < K
\end{align*}
where the norms are defined by $g$, then 
\[
\| (L_h-L_{\tilde h})(s)\|_{C^{m-2,\eta}} \leq C \| h- \tilde{h} \|_{C^{m,\eta}} \| s\|_{C^{m,\eta}}
\]
for all symmetric 2-tensors $s \in C^{m,\eta}$. (Here $L_h$ and $L_{\tilde{h}}$ are both defined  using $g$ as the reference metric for the Bianchi gauge.)
\end{lemma}

This is again a standard result and we omit the proof. When there is no ambiguity we write $C^{m,\eta}$ for the space of {symmetric bilinear forms} of regularity $C^{m,\eta}$, and the H\"older norms in Lemma \ref{Lipschitz-continuity-L} are measured with respect to $g$ (see Definition \ref{Holder-definition} for the explicit definition of the H\"older norms used in this article). Lemma~\ref{Lipschitz-continuity-L} implies that $\Phi_g \colon C^{m,\eta} \to C^{m-2,\eta}$ is a continuously differentiable map of Banach spaces. (Strictly speaking, the domain of $\Phi_g$ is the open subset of $C^{m,\eta}$ consisting of positive definite sections.) 

We next recall another important fact about Bianchi gauge: at least in the case of negative Ricci curvature, zeros of $\Phi_g$ are precisely Einstein metrics. To see this, one computes that
\begin{equation}
2 B_h \circ \div_h^* =  \nabla^*_h\nabla_h - \Ric(h)
\label{Bianchi-Weitzenbock}
\end{equation} 
In particular, when $\Ric(h)$ is negative, $B_h \circ \div^*_h$ is an isomorphism. From this, the next result follows easily.

\begin{lemma}\label{zeros-are-Einstein}
Let $(X,g)$ be a closed Riemannian manifold and $h$ a second metric on $X$ with $\Ric(h) < 0$. If $\Phi_g(h)= 0$ then in fact $\Ric(h) = -(n-1)h$ and $B_g(h) = 0$.
\end{lemma}
\begin{proof}
Since $B_h(\Ric_h)=0=B_h(h)$ the fact that $B_h(\Phi_g(h))=0$ implies $B_h(\div_h^*B_g(h))=0$. Equation \eqref{Bianchi-Weitzenbock} and integration by parts then implies that $B_g(h)=0$ and so $\Ric_h=-(n-1)h$. 
\end{proof}

\subsection{The linearisation is invertible}

We return to our sequence $(X_k, g_k)$ of approximately Einstein  metrics, constructed in \S\ref{approximate-solution}. Write $\Phi_k = \Phi_{g_k}$ for the Einstein operator in Bianchi gauge relative to $g_k$. By Proposition~\ref{prop-approx-solution}, we have that $\Phi_k(g_k) = O(U_k^{1-n})$. We would like to apply the inverse function theorem to $\Phi_k$ to show that for sufficiently large $k$ there is a metric $h$ near to $g_k$ with $\Phi_k(h)= 0$. Since $g_k$ has negative curvature, the same will be true of $h$ and so, by Lemma~\ref{zeros-are-Einstein}, $h$ will be the Einstein metric we seek. {We first show that} the linearisation $L_g$ of $\Phi_k$ at $g$ is invertible for $g$ on a definite neighbourhood of $g_k$ whose diameter is bounded below independently of $k$. 
\begin{proposition}\label{invertible}
There exist constants $\delta>0$ and $C>0$ such that for all sufficiently large $k$, if $g$ is a Riemannian metric on $X_k$ with 
\[
\| g - g_k \|_{C^2} \leq \delta
\]
then, for any $C^2$ symmetric bilinear form $s$,
\begin{equation}
\int_{X_k} \left\langle L_{g}(s),s \right\rangle_g \dvol_{g}
\geq
C\int_{X_k} |s|_g^2 \dvol_{g}
\label{coercivity}
\end{equation}
It follows that for any $m\geq 2$ and $0< \eta <1$, the linearisation $L_{g} \colon C^{m,\eta} \to C^{m-2,\eta}$ is an isomorphism.
\end{proposition}

In the statement of Proposition~\ref{invertible} it is implicit that the H\"older norm is taken with respect to the metric $g_k$. Throughout the proof we use the fact that, provided $\delta$ is small enough, the $C^0$ norms defined by $g$ and $g_k$ are equivalent uniformly in $k$. We will switch between them without further comment. Our proof is an adaptation of an argument {due to Koiso \cite{koiso} to our setting. Koiso's Theorem showed  that if $h$ is an Einstein metric of negative sectional curvature on a compact manifold $M$, then 
\[
\int_M \left\langle L_h(s),s \right\rangle \geq \frac{n-2}{2} K \int_M |s|^2 \dvol
\]
where  $-K$ is the supremum of the sectional curvatures of $h$. Here, $L_h$ is the linearised Einstein operator in Bianchi gauge relative to $h$ itself. In our situation, there are two small additional complications. Firstly, $g_k$ is merely close to Einstein; secondly we linearise at an arbitrary metric close to $g_k$. In particular, when $g \neq g_k$, the linearised operator at $g$ is not self-adjoint anymore. A first lemma is as follows:}

\begin{lemma}\label{step1-invertibility}
There exist constants $\delta_0>0$ and $C>0$ such that for all $0<\delta \leq \delta_0$ and all sufficiently large $k$, if $g$ is a Riemannian metric on $X_k$ with 
\[
\|g- g_k \|_{C^2} \leq \delta
\]
then, for any $C^2$ symmetric bilinear form $s$,
\begin{multline}\label{intermediate-bound}
\int_{X_k} \left\langle L_g(s),s \right\rangle_g \dvol_{g}
\geq
\left( \frac{1}{2} - C\delta \right) \int_{X_k} |\nabla s|^2_g \dvol_g
-
C\delta \int_{X_k} |s|^2_g \dvol_g\\
-
\int_{X_k}\left\langle \Rm_g(s),s \right\rangle_g \dvol_g
\end{multline}
(where $\nabla$ is the Levi-Civita connection of $g$).
\end{lemma}

\begin{proof}
From the expression~\eqref{linearised-Einstein-general} for $L_g$, we can write
\[
\int_{X_k} \left\langle L_{g}(s),s \right\rangle_g \dvol_{g}
=
\frac{1}{2}\int_{X_k} |\nabla s|^2_g \dvol_g
+ 
I_1 + I_2 + I_3 + I_4 + I_5
\]
where
\begin{align*}
I_1 
	&= 
		\int_{X_k} 
		\left\langle \frac{1}{2}\left( 
		\Ric_g \circ s + s \circ \Ric_g \right) 
		+ (n-1)s,s \right\rangle \dvol_g\\
I_2
	&=
		\int_{X_k} \left\langle \div_g^*(\div_{g_k} - \div_{g})(s), s \right\rangle \dvol_g\\
I_3
	&=
		\int_{X_k} {\frac12}\left\langle \nabla \diff (\tr_g s - \tr_{g_k} s), s \right\rangle \dvol_g\\
I_4
	&=
		\int\left\langle  B_{g_k}(g) * \nabla s, s \right\rangle \dvol_g\\
I_5
	&=
		- \int_{X_k} \left\langle \Rm_g(s), s \right\rangle\dvol_g
\end{align*}
(where $\nabla$ denotes the Levi-Civita connection of $g$.)

{Since the metrics $g_k$ have uniformly bounded curvature tensors we get that for all large~$k$,
\[
\| \Ric_g - \Ric_{g_k} \|_{C^0} \leq C \delta
\]
(see, for example, Lemma~2.8 of~\cite{fine} for a proof).
Combining the latter with Proposition~\ref{prop-approx-solution}, for all $k$ sufficiently large, we get that there is a constant $c_1>0$ such that
\begin{equation}\label{Ricci-bound}
\| \Ric_g + (n-1)g \|_{C^0}
	\leq 
	c_1\delta
\end{equation}
As a consequence: 
\begin{equation}\label{I1}
|I_1| \leq c_1 \delta \int_{X_k} |s|^2_g \dvol_g
\end{equation}}
{Next, there is a constant $c>0$ such that at all points
\[
\left| 
\left(\div_{g_k} - \div_g\right)(s)
\right|_g
	\leq 
		c \delta \left( |s|_g + |\nabla s|_g \right)
\]
Integrating by parts and using Young's inequality then yields:
\begin{equation}\label{I2}
|I_2| \leq c_2 \delta \left( \int_{X_k} |s|^2_g + |\nabla s|^2_g  \right) \dvol_g 
\end{equation}
for some constant $c_2$, provided $\delta>0$ is small enough.} In exactly the same way, we see that there are constants $c_3,c_4$ such that
\begin{align}
|I_3| &\leq c_3 \delta \left( \int_{X_k} |s|^2_g + |\nabla s|^2_g  \right) \dvol_g \label{I3}\\
|I_4| &\leq c_4 \delta \left( \int_{X_k} |s|^2_g + |\nabla s|^2_g  \right) \dvol_g \label{I4}
\end{align}
For \eqref{I4} we used that $B_{g_k}(g_k) = 0$ and thus 
\[
\left|
B_{g_k}(g) * \nabla s
\right|_g
\leq 
	c\delta |\nabla s|_g
\]
Finally, putting \eqref{I1}, \eqref{I2}, \eqref{I3}, \eqref{I4} together proves the result.
\end{proof}

\begin{proof}[Proof of Proposition~\ref{invertible}]

{First, a Weitzenb\"ock formula (see e.g. equation (1.5.6) in~\cite{biquard}) shows that for any $C^2$ symmetric bilinear form $s$, 
\begin{equation} \label{step2-invertibility}
\int_{X_k} |\nabla s|^2_g \dvol_g
	\geq
		\left( n-1- c \delta \right) \int_{X_k} |s|^2_g \dvol_g
		+
		\int_{X_k} \left\langle \Rm_g(s),s \right\rangle_g \dvol_g
\end{equation}
(where $\nabla$ is the Levi-Civita connection of $g$). Putting Lemma~\ref{step1-invertibility} and~\eqref{step2-invertibility} together we then see that}
\begin{multline}\label{second-intermediate-bound}
\int_{X_k} \left\langle L_g(s),s \right\rangle_g \dvol_{g}
\geq
\left( - \frac{1}{2} - C\delta \right)\int_{X_k}\left\langle \Rm_g(s),s \right\rangle_g \dvol_g\\
+
\left( \frac{n-1}{2} - C\delta \right)\int_{X_k} |s|^2_g \dvol_g
\end{multline}
We work at a point $p$ and write $\lambda_i$ for the eigenvalues of $s$ at that point. Then, in an orthonormal eigenbasis for $s$,
\[
\left\langle \Rm_{g}(s),s  \right\rangle 
	= 
		s^{ab}s^{pq}R_{apbq}\\
	=
		\sum_{i \neq j} \lambda_i \lambda_j K_{ij}
\]
where $K_{ij}$ is the sectional curvature of $g$ in the plane spanned by the $\lambda_i$ and $\lambda_j$ eigendirections of $s$. By Proposition~\ref{prop-approx-solution}, the sectional curvatures of $g_k$ are negative, uniformly in $k$. Since $g$ is close to $g_k$ in $C^2$, it also has uniformly negative sectional curvatures. I.e., there is a constant $K>0$ such that $K_{ij} <  -K$ for all $k$ (and all points). Moreover, by~\eqref{Ricci-bound}, we have control over the following sum of the sectional curvatures:
\[
\sum_{j} K_{ij} = \Ric_{g}(e_i,e_i) \geq 1-n -c\delta
\]
where $e_i$ is a unit eigenvector for the $\lambda_i$-eigenvalue of $s$. 
{A straightforward adaptation of Koiso's argument (\cite{koiso}, see also Lemma $12.71$ in \cite{besse}) then gives:}
\[
\int_{X_k} \left\langle \Rm_g(s),s \right\rangle_g \dvol_g 
	\leq
		\left(n-1 - (n-2)K + c \delta\right)\int_{X_k} |s|^2_g \dvol_g
\]
{With \eqref{second-intermediate-bound} this yields the desired result provided $\delta$ is taken small enough. }
\end{proof}

\subsection{Weighted H\"older spaces} \label{holder-norm-definition}

The crux to applying the inverse function theorem to find a zero of $\Phi_{k}$ is to obtain uniform control over the inverse $L_{g}^{-1}$. Proposition~\ref{invertible} shows that the lowest eigenvalue of $L_{g}$ is uniformly bounded away from zero and this immediately gives good control of the inverse in $L^2$. This is not sufficient for our purposes, however. The volume of $(X_k,g_k)$ grows rapidly with $k$, so much so that even though we have strong pointwise control of $\Phi_{g_k}(g_k)$ it does not even imply that $\Phi_{g_k}(g_k)$ tends to zero in $L^2$. Instead we work in H\"older spaces. Moreover, since the diameter of $(X_k, g_k)$ tends to infinity, weighted H\"older spaces are required.

We begin with a word on the definition of \emph{unweighted} H\"older spaces. 

\begin{definition}\label{Holder-definition}
Let $(X,g)$ be a compact Riemannian manifold. Write $\rho(g)$ for the conjugacy radius of $g$ and fix $\rho_0 < \rho(g)$. Given $x \in X$ write $\exp_x \colon T_x X \to X$ for the exponential map, which is a local diffeomorphism on the ball $B(0,\rho_0) \subset T_xX$. Let $s$ be a tensor field on $X$. Then $\exp^*_x(s)$ is a tensor field on the Euclidean vector space $T_xX$ and we can use the Euclidean metric to define the H\"older coefficient of $s$ near $x$: 
\[
\left[ s \right]_{\eta, x}
	\defeq
		\sup_{p\neq q \in B(0,\rho_0)} 
		\frac{\left| \exp_x^*(s)(p) - \exp_x^*(s)(q)\right|}{|p-q|^\eta}
\]
We then take the supremum over all points $x$ and combine with derivatives to take the full H\"older norm:
\[
\| s\|_{C^{m,\eta}}
	\defeq
		\sum_{j \leq m}\sup_{x\in X} |\nabla^j s(x)| + \sup_{x\in X} [\nabla^m s]_{\eta,x}
\]
\end{definition}

This definition of the H\"older norm is well adapted to studying sequences $(X_k,g_k)$ for which there is a uniform bound for the curvature and its derivatives: for all $m\in \N$ there exists $C>0$ such that $\| \Rm(g_k)\|_{C^m} \leq C$. 
 
Our sequence of approximately Einstein metrics have uniform $C^m$ bounds on $\Rm(g_k)$, thanks to part~\ref{curvature-bounded} of Proposition~\ref{prop-approx-solution}.

We now move to the weighted norms. We begin by defining the weight function $w$.

\begin{lemma}\label{weight-function}
For all large $k$, there exists a smooth function $w \colon X_k \to \R$ such that
\begin{enumerate}
\item In the region $\{ u < \frac{1}{2}U_{k,\max}\}$, $w=u$.
\item Outside the region $\{ u < U_{k,\max}\}$, $w = U_{k,\max}$.
\item For each $m$, there is a constant $C$ (not depending on $k$), such that $|\nabla^mw| \leq C|w|$ (where the norm is taken with $g_k$).
\end{enumerate}
\end{lemma}

\begin{proof}
{Near each component of the branch locus we have a distinguished coordinate $u$, used in the construction of the model metric, given by \eqref{model-ansatz} and \eqref{V-definition}. It is defined for $u_a \leq u < U_{k,\max}$ (where $a$ is chosen so that the metric on $M_k$ has cone angles $2\pi/l$). We extend this to a function $w \colon X_k \to \R$ by setting it to be constant, equal to $U_{k,\max}$ outside the region  $\{ u < U_{k,\max}\}$. We then modify it in the region $ \frac{1}{2} U_{k,\max} \le u \le U_{k,\max}$ to make this extension smooth. The estimate on the derivatives of $w$ follows from \eqref{LC-model}.} 
\end{proof}

\begin{definition}
Let $m \in \N$, $0 \leq \eta <1$ and $\alpha >0$. Given a symmetric bilinear form $s$, we define the weighted H\"older norm of $s$ to be
\[
\| s\|_{C^{m,\eta}_\alpha} \defeq \| w^\alpha s\|_{C^{m,\eta}}
\]
where $w$ is the weight function of Lemma~\ref{weight-function} and the norm is taken for $g_k$. 
\end{definition}

{The following result} is typical in the use of weight functions and we refer for instance to \cite{biquard}, \S 3.8. Recall that the uniform control on sectional curvatures of $g_k$ gives a uniform lower bound on the conjugacy radius $\rho(g_k) \geq \rho_0$ of the manifolds $(X_k,g_k)$. 

\begin{lemma}\label{local-norm-control}
Let $m \in \N$, $0\leq \eta <1$ and $0<\rho \leq \rho_0$. Then there exists a constant $C= C(m,\eta,\rho)>0$ such that for any $x \in X_k$ and any {symmetric bilinear form} $s$ of regularity $C^{m,\eta}$, we have
\begin{equation}\label{norms-on-balls}
\frac{1}{C} w(x)^\alpha \| s\|_{C^{m,\eta}(B_x(\rho))}
	\leq
		\| s \|_{C^{m,\eta}_\alpha(B_x(\rho))}
	\leq 
		C w(x)^\alpha \| s\|_{C^{m,\eta}(B_x(\rho))}
\end{equation}
where $B_x(\rho) \subset (X_k, g_k)$ denotes the geodesic ball centred at $x$ with radius $\rho$. In particular, $C$ is independent of both $x$ and $k$.
\end{lemma}
An easy consequence is the following:

\begin{corollary}\label{weighted-controls-unweighted}
For any $m$, $0 \leq \eta < 1$ and $\alpha >0$ there is a constant $C$, independent of $k$, such that for all $s \in C^{m,\eta}$,
\[
\| s \|_{C^{m,\eta}} \leq C \| s\|_{C^{m,\eta}_\alpha}
\]
\end{corollary}
\begin{proof}
This follows from Lemma~\ref{local-norm-control}, taking the supremum over $x$, together with the fact that $w \geq u_a>0$, a lower bound which is independent of $k$.
\end{proof}

{As a consequence of the definition of the weighted norms and of Proposition \eqref{prop-approx-solution} we have:}

\begin{lemma}\label{weighted-error}
For all integers $m\geq0$ and real numbers $0\leq \eta < 1$, there is a constant $A$ such that
\[
\| \Ric(g_k) + (n-1) g_k\|_{C^{m,\eta}_\alpha}
\leq A U^{1-n+\alpha}_k
\]
\end{lemma}

We conclude this section {by stating without proof} the weighted analogues of standard elliptic estimates.

\begin{lemma}
\label{weighted-Lipschitz-continuity}
Fix an integer $m\geq 2$, and real numbers $0 \leq \eta <1$ and $\alpha >0$. There are constants $\delta, C>0$, independent of $k$, such that if $g$ and $h$ are Riemannian metrics on $X_k$ with 
\[
\| g - g_k\|_{C^{m+2, \eta}}, \|h - g_k\|_{C^{m+2,\eta}} \leq \delta
\]
then for all symmetric bilinear forms $s$ of regularity $C^{m+2,\eta}_\alpha$ we have
\[
\| \left( L_g - L_h \right)(s) \|_{C^{m,\eta}_\alpha}
	\leq
		C \| g - h\|_{C^{m+2,\eta}} \|s\|_{C^{m+2,\eta}_\alpha}
\]
(where all norms are taken with respect to $g_k$).
\end{lemma}

\begin{remark}
Note that by Corollary~\ref{weighted-controls-unweighted}, we are free to replace the unweighted norms on the metrics in this result by weighted ones (at the expense of shrinking $\delta$). We will frequently do this in applications of this result.
\end{remark}

\begin{lemma}
\label{weighted-elliptic}
For any integer $m$ and real numbers $0 < \eta < 1$ and $\alpha \geq 0$ there are constants $\delta,C >0$, independent of $k$, such that if $g$ is a Riemannian metric on $X_k$ with 
\[
\| g- g_k \|_{C^{m+2,\eta}_{\alpha}} \leq \delta
\]
then for all  {symmetric bilinear forms} $s$ of regularity $C^{m+2,\eta}$, we have
\[
\| s \|_{C^{m+2,\eta}_\alpha}
	\leq 
		C \left( 
 		\| L_{g}(s)\|_{C^{m,\eta}_\alpha} + \| s\|_{C^0_\alpha}
		\right)
\]
(where all norms are taken with respect to $g_k$).
\end{lemma}
Fix an integer $m\geq 2$, and real numbers $0 < \eta <1$ and $\alpha >0$. We claim that there are constants $\delta >0$ (independent of $k$) and $C_k>0$ (depending on $k$) such that if $g$ is a Riemannian metric on $X_k$ with 
\[
\| g-g_k\|_{C^{m+2,\eta}_\alpha} \leq \delta
\]
then for all symmetric bilinear forms $s$ of regularity $C^{m,\eta}$ we have
\begin{equation} \label{naive-inverse-bound}
\| s \|_{C^{m+2,\eta}_\alpha} \leq C_k \| L_g(s)\|_{C^{m,\eta}_\alpha}
\end{equation}

This follows from the previous results of this section and is essentially the standard contradiction argument used to remove the $C^0$ term in the elliptic estimate, based on the fact that $L_g$ is invertible by Proposition \ref{invertible}. Note that the constant in \eqref{naive-inverse-bound} depends on $k$, because the contradiction argument must be carried out on each $X_k$ separately.

\subsection{The proof assuming a key estimate}

We now explain how to perturb $g_k$ to an Einstein metric, assuming temporarily one critical estimate, Theorem~\ref{key-estimate} below. We will prove this estimate in the case $\dim X_k=4$ in the following section.

The first step in the proof is to apply a version of the inverse function theorem to $\Phi_k$ with uniformity in $g$, if not in $k$. We state the result here:

\begin{proposition}\label{naive-IFT}
Fix an integer $m \geq 0$ and real numbers $0<\eta <1$ and $\alpha >0$. There exist constants $\delta >0$ (independent of $k$) and $r_k>0$ (depending on $k$) such that if $g$ is a Riemannian metric on $X_k$ with 
\[
\| g - g_k\|_{C^{m+2,\eta}_\alpha} \leq \delta
\] 
then $B(\Phi_k(g), r_k) \subset C^{m,\eta}_\alpha$ is contained in the image of $\Phi_k$ and there is a differentiable map 
\[
\Psi_k \colon B(\Phi_k(g), r_k) \to C^{m+2,\eta}_\alpha
\]
inverting $\Phi_k$ on a neighbourhood of $g \in C^{m+2,\eta}_\alpha$. 
\end{proposition}
\begin{proof}
This is just an application of the inverse function theorem to $\Phi_k$ at $g$, where the invertibility is given by Proposition~\ref{invertible}.  The inverse function theorem provides a quantitative estimate on the radius of a ball centred at $\Phi_k(g)$ which is contained in the image of $\Phi_k$. By the uniform Lipschitz continuity of $\Phi_k$ (Lemma  \ref{weighted-Lipschitz-continuity}) this radius depends only on the square inverse of the operator norm of  $L_g = \textrm{d} \Phi_k (g)$. By \eqref{naive-inverse-bound} we can choose such a bound to only depend on~$k$.
\end{proof}

Of course, this is far from enough to prove the existence of an Einstein metric. Whilst $\Phi_k(g_k)$ tends to zero as $k$ tends to infinity, the radius $r_k$ may tend to zero even quicker. To remedy this problem we will use a much sharper estimate on $L_g^{-1}$. \emph{It is at this point our argument requires $\dim X_k =4$.}

As a matter of notation, write 
\[
S_k 
	= 
		\left\{
			\frac{1}{2}U_k \leq u \leq U_k 
		\right\}
\]
Recall that $\Ric(g_k) + 3g_k$ is supported in $S_k$. Recall also that until now, our choice of gluing parameter $U_k$ has only had to satisfy the requirements that $U_k \to \infty$ and $U_k \leq \frac{1}{2}U_{k,\max}$. We will need to be more careful in our choice of $U_k$ in order to prove the estimate we want.

\begin{theorem}\label{key-estimate}
Let $\dim X_k = 4$.  There is a choice $(U_k)$ of gluing parameters such that for the corresponding approximately Einstein manifolds $(X_k,g_k)$ the following holds. For any integer $m \geq 1$ and real number $0< \eta <1$ there exists real numbers $0<\alpha < 3$ and $\delta >0$ and a sequence $(\epsilon_k)$ of positive real numbers, with $\epsilon_k \to 0$ as $k \to \infty$ which have the following property. For all large $k$, if $g$ is a {Riemannian} metric on $X_k$ with
\[
\| g- g_k\|_{C^{m+2,\eta}_\alpha} \leq \delta
\]
then for any symmetric bilinear tensor $s \in C^{m+2,\eta}_\alpha$, with $L_g(s)$ supported in $S_k$ we have
\[
\| s\|_{C^0_\alpha} \leq \epsilon_k U^{3-\alpha}_k\| L_g(s)\|_{C^{m,\eta}_\alpha}
\]
\end{theorem}

It is crucial in Theorem~\ref{key-estimate} that we restrict attention to those $s$ with $L_g(s)$ supported in $S_k$. It seems that the sought-after estimate will not hold otherwise. Whilst our proof does not extend to arbitrary dimensions, it seems plausible that the analogous estimate could hold in dimension $n$ (with the power $U_k^{n-1+\alpha}$ on the right-hand side). This would then imply the existence of Einstein metrics for these higher dimensional Gromov--Thurston manifolds.

Theorem~\ref{key-estimate} is the core of the analysis of this paper and we prove it in the following section. For the remainder of this section we show how this refined estimate proves the existence of an Einstein metric.

\begin{proof}[Proof of Theorem~\ref{main-theorem}, assuming Theorem~\ref{key-estimate}]
Let $\gamma(t) = (1-t)\Phi_k(g_k)$. Proposition~\ref{naive-IFT} gives a smooth path of Riemannian metrics $g(t)$ solving $\Phi_k(g(t)) = \gamma(t)$ for 
\[
0\leq t < \frac{r_k}{\| \Phi_k(g_k)\|_{C^{m, \eta}_\alpha}}
\]
We will show this path can be extended up to $t=1$, and then $g(1)$ is the Einstein metric we seek. 

Let $\delta>0$ be small enough so that Lemma~\ref{weighted-elliptic}, Proposition~\ref{naive-IFT} and Theorem~\ref{key-estimate} all apply simultaneously. Write $B_\delta \subset C^{m+2,\eta}_\alpha$ for the ball of radius $\delta$ centred at $g_k$. Consider the set
\[
T 
	= \left\{ 
		\tau>0 : \text{ there is a differentiable map }g \colon [0,\tau]\to B_\delta \text{ with }\Phi_k(g(t)) = \gamma(t), g(0)=g_k
		\right\}
\]
Let $\sigma = \sup T$. By Proposition~\ref{naive-IFT} we know that 
\[
\sigma \geq \frac{r_k}{\|\Phi_k(g_k)\|_{C^{m,\eta}_\alpha}}
\]
We will show that $\sigma <1$ gives a contradiction. Consider
\[
\tau = \sigma - \frac{r_k}{2\| \Phi_k(g_k)\|_{C^{m,\eta}_\alpha}}
\]
We have $0 <\tau < \sigma$ and so the path $g(t)$ exists on $[0,\tau]$ and stays inside $B_\delta$. By Proposition~\ref{naive-IFT}, $\Phi_k$ is a local diffeomorphism at $g(\tau)$ and its image contains the ball of radius $r_k$ centred at $\gamma(\tau)$. In particular it contains $\gamma(t)$ for 
\[
t \in \left[ \tau, \sigma + \frac{r_k}{2\|\Phi_k(g_k)\|_{C^{m,\eta}_\alpha}}\right]
\]
So we can actually extend $g(t)$ smoothly to solve $\Phi_k(g(t)) = \gamma(t)$ for values of $t$ slightly larger than $\sigma$. The crux is to show that in doing so we do not leave $B_\delta$. 

To prove this, differentiate $\Phi_k(g(t)) = (1-t)\Phi_k(g_k)$ with respect to $t$ to get
\[
L_{g(t)} (g'(t)) = -\Phi_k(g_k)
\]
For $t \in [0,\sigma)$, $g(t) \in B_\delta$ and so, for these times, we can apply the elliptic estimate Lemma~\ref{weighted-elliptic} and Theorem~\ref{key-estimate}. This, together with the error estimate Lemma~\ref{weighted-error}, gives
\begin{align*}
\| g'(t) \|_{C^{m+2,\eta}_\alpha}
	&\leq 
		C \left( 
		\| \Phi_k(g_k)\|_{C^{m,\eta}_\alpha}
		+
		\| g'(t) \|_{C^0_\alpha}
		\right)\\
	&\leq
		AC\left(1 + \epsilon_kU_k^{n-1-\alpha}\right) U_k^{1-n+\alpha}
\end{align*}
This bound tends to zero as $k$ tends to infinity and so for all large $k$ we have $\|g'(t)\|_{C^{m+2,\eta}_\alpha} < \delta$. Integrating this from $t=0$ to $t=\sigma$, we see that
\[
\| g(\sigma) - g_k\|_{C^{m+2,\eta}_\alpha} \leq \sigma \delta
\]
So the assumption that $\sigma <1$ means $g(\sigma) \in B_\delta$ and hence $g(t) \in B_\delta$ for $t$ slightly larger that $\sigma$. This is a contradiction with the fact that $\sigma = \sup T$.

We write $g$ for the Einstein metric on $X_k$ found in this way. To check that the sectional curvatures of $g$ are negative, recall that there is a constant $c>0$ such that the sectional curvatures of the approximate solution $g_k$ all satisfy $\sec(g_k) \leq -c$. By construction, our Einstein metric $g$ is of the form $g = g_k + s_k$ where $\|s_k\|_{C^{2,\eta}} \to 0$. From this it follows that the sectional curvatures of $g$ satisfy $\sec(g) \leq -c/2$.

We now check that $g$ is not simply locally homogeneous. There is a constant $b>0$ such that the model metric of Propositon~\ref{model-is-Einstein} has at least one sectional curvature at finite distance from the branch locus, which satisfies $\sec \geq -1 +b$. (This follows from the explicit form of the sectional curvatures given in the proof of Lemma~\ref{model-negatively-curved}.) Since  $g$ is a $C^2$-small perturbation of this metric near the branch locus, it must have a a sectional curvature which satisfies $\sec \geq -1 + b/2$. However, the approximate solution $g_k$ is genuinely hyperbolic at large distances and so at these distances all sectional curvatures of the Einstein metric $g$ satisfy $\sec < -1 + b/2$. It follows that $g$ near the branch locus is not locally isometric to $g$ at large distances and hence $g$ is not locally homogeneous.
\end{proof}

{\begin{remark}
In the proof of Theorem \ref{key-estimate}, the weighted control of $g_k-g$ plays a crucial role, as can be seen for instance in the proof of Lemma \ref{Ak-integral-bound}. We stated Lemma \ref{weighted-elliptic} and Proposition \eqref{naive-IFT} with a weighted control on the metrics to be coherent, but they remain true if one only assumes that $\| g - g_k\|_{C^{m+2,\eta}} \leq \delta$. \end{remark}}

\section{Proving the key estimate}\label{proof-of-key-estimate}

In this section we prove Theorem~\ref{key-estimate}, which completes the proof of Theorem~\ref{main-theorem}. We quickly recall some of our notation. $U_{k,\max} = \cosh(\frac12 i(M_k))$, where $i(M_k)$ is the injectivity radius of $M_k$ with the hyperbolic metric; the gluing is carried out in the region $\frac{1}{2}U_k \leq u \leq U_k$, where $U_k < \frac{1}{2}U_{k,\max}$. In the course of the proof, it will be important how we choose the gluing parameter $U_k$. For now, we stipulate only that $U_k \to \infty$ whilst $U_k/U_{k,\max} \to 0$. The precise choice will be made later.

The proof is by contradiction and so we assume Theorem~\ref{key-estimate} is false. I.e.:

\begin{assumption*}
Let $m \ge 1$. Let $\alpha>0$, $\delta_0>0$, $(\epsilon_k)$ be a sequence of positive real numbers with $\epsilon_k \to 0$ and let $N_0 \in \N$. Then there exist  $k_0 \geq N_0$, $\tilde{g}_{k_0}$ a metric on $X_k$ with
\[
\| \tilde{g}_{k_0} - g_{k_0}\|_{C^{m+2,\eta}_\alpha} \leq \delta_0
\]
and $s_{k_0} \in C^{m+2,\eta}_\alpha$ with $L_{\tilde{g}_{k_0}}(s_{k_0})$ supported in $S_{k_0}$, such that
\[
\| s_{k_{0}}\|_{C^0_\alpha} > \epsilon_{k_0} U^{3-\alpha}_{k_0} \| L_{\tilde{g}_{k_0}}(s_{k_0})\|_{C^{m,\eta}_\alpha}
\]
\end{assumption*}

We fix $\alpha$ and take $\epsilon_k = U_k^{-p}$ for some small positive number $p$. Both $\alpha$ and $p$ will be determined in the course of the proof. We now apply our hypothesis with $\delta_0$ replaced by a sequence {$\delta_q>0$} with ${\delta_q} \to 0$ and $N_0$ replaced by a sequence ${N_q} \in \N$ with {$N_q \to \infty$}. This gives a sequence ${\tilde{g}_{k_q}}$ of metrics and symmetric bilinear forms ${s_{k_q}}$ on ${X_{k_q}}$  such that the conclusions of the hypothesis are satisfied. To ease the notation, we pass to this subsequence and drop the {$q$} subscript. This leads to a sequence $(\tilde{g}_k)$ of metrics with 
\begin{equation}\label{metrics-converge}
\| \tilde{g}_{k} - g_{k}\|_{C^{m+2,\eta}_\alpha} \to 0
\end{equation}
as $k \to \infty$, and a sequence $(s_k)$ of symmetric bilinear forms for which $L_{\tilde{g}_{k}}(s_{k})$ is supported in $S_{k}$ and 
\begin{equation}\label{bound-to-be-violated}
\| s_{k}\|_{C^{0}_\alpha} > U_{k}^{3 - \alpha - p}\| L_{\tilde{g}_{k}}(s_{k})\|_{C^{m,\eta}_\alpha}
\end{equation}

We will prove that \eqref{bound-to-be-violated} actually never holds, giving our contradiction. To do this, for each $k$ we pick  $x_k \in X_k$ at which
\[
w^\alpha(x_k) |s_k(x_k)|_{g_k} = \| s_k \|_{C^0_\alpha}
\]
where $w$ is the weight constructed in Lemma \ref{weight-function}.

First, a word on notation. Given sequences $(p_k)$ and $(q_k)$ of real numbers, {we} write $p_k \lesssim q_k$ to mean that there is a constant $C>0$ such that for all $k$, $p_k \leq C q_k$. In a chain of such inequalities, $p_k \lesssim q_k \lesssim r_k$, the constant $C$ may change, but will always be independent of $k$.

The first step in the proof is a preliminary lemma, showing that the $C^0_\alpha$-norm of $L_{\tilde{g}_k}(s_k)$ gives control of $s_k$ in $W^{1,2}$. It is at this point that the crucial bound on the volume of the branch locus, derived in \S\ref{spin-GT}, enters the analysis. Recall part~3 of Proposition~\ref{hyperbolic-sequence}, which says that, in arbitrary dimension,
\[
\vol(\Sigma_k) \leq A \exp \left( \frac{n^2 - 3n + 6}{4} i(M_k) \right)
\]
In our case, $n=4$. By definition of $U_{k,max}$ we deduce that
\begin{equation}
\vol(\Sigma_k) \lesssim U_{k,\max}^5
\label{volume-branch}
\end{equation}

\begin{lemma}\label{preliminary-bounds}
We have 
\begin{align}
\| L_{\tilde{g}_k}(s_k)\|_{L^2} 
	&
	\lesssim 
		U_{k,\max}^{\frac{5}{2}} U_k^{\frac{3}{2} -\alpha}\| L_{\tilde{g}_k}(s_k)\|_{C^0_\alpha}
	\label{L2-bound-Lsk}\\
\| s_k \|_{L^2} + \| \nabla s_k \|_{L^2}
	&\lesssim
		U_{k,\max}^{\frac{5}{2}} U_k^{\frac{3}{2} -\alpha}\| L_{\tilde{g}_k}(s_k)\|_{C^0_\alpha}
		\label{L2-bound-sk}
\end{align}
Both the $L^2$ and H\"older norms here are taken with respect to the metric $g_k$.
\end{lemma}
\begin{proof}
$L_{\tilde{g}_k}(s_k)$ is supported in $S_k$, so
\[
\| L_{\tilde{g}_k}(s_k)\|_{L^2} \lesssim \vol(S_k)^{\frac{1}{2}} \| L_{\tilde{g}_k}(s_k)\|_{C^0}
\]
But, by definition of $S_k$, at all points of $S_k$ the weight function $w$ satisfies $w \gtrsim U_k$ from which we have $\| L_{\tilde{g}_k}(s_k)\|_{C^0} \lesssim U_k^{-\alpha} \| L_{\tilde{g}_k}(s_k)\|_{C^0_\alpha}$. We have $\vol(S_k) \lesssim U^3_k \vol(\Sigma_k)$. Now~\eqref{volume-branch} implies~\eqref{L2-bound-Lsk}.

From here, Proposition~\ref{invertible} gives $\| s_k \|_{L^2} \lesssim \| L_{\tilde{g}_k} (s_k) \|_{L^2}$. Proposition~\ref{invertible} gives this with the $L^2$-norms defined by $\tilde{g}_k$ but by \eqref{metrics-converge} these norms are equivalent to those defined by $g_k$. By ~\eqref{L2-bound-Lsk} this proves \eqref{L2-bound-sk} for $\| s_k \|_{L^2}$.

We now use \eqref{intermediate-bound}, which gives
\[
\int_{X_k} |\nabla s_k|^2_{\tilde{g}_k} \dvol_{\tilde{g}_k} 
	\lesssim 
		\int_{X_k} \left\langle L_{\tilde{g}_k}(s_k), s_k \right\rangle_{\tilde{g}_k}
		+
		\int_{X_k} | s_k|^2_{\tilde{g}_k} \dvol_{\tilde{g}_k},
\]
from which \eqref{L2-bound-sk} follows by the previous arguments (we have used here the fact that $\Rm_{\tilde{g}_k}$ is bounded uniformly in $k$.) \end{proof}

At this point we divide the argument into three separate cases.
\begin{enumerate}
\item 
There exists a constant $C>0$ such that, after passing to a subsequence, for all large $k$ we have
\[
w(x_k) \geq \frac{1}{C} U_{k,\max}
\]
The points $x_k$ are further and further from the branch locus. Moreover, since we choose the gluing distance with $U_k/U_{k,\max} \to 0$, for large $k$ the points $x_k$ lie in the region of $X_k$ where $g_k$ is genuinely hyperbolic. 
\item
There exists a constant $C$ such that, after passing to a subsequence, 
\[
w(x_k) \leq C
\]
The points $x_k$ remain at bounded distance from the branch locus $\Sigma_k \subset X_k$ and so lie in the region where $g_k$ is given by the model Einstein metric of \S\ref{approximate-solution}.
\item
The remaining possibility is that $w(x_k) \to \infty$ and $w(x_k)/U_{k,\max} \to 0$. In this case the points $x_k$ live in a region where the model coordinate system near the branch locus makes sense, but they are moving further and further from the branch locus. 
\end{enumerate}
We will treat each of these cases separately, but each time the argument follows similar lines. We translate the problem onto a non-compact space (either $\H^4$ or the model metric of \S\ref{approximate-solution}). We  use a Green's representation formula in this non-compact space to give an expression for $s_k(x_k)$. We then prove estimates for the Green's operator. In cases~1 and~2 these are weighted integral estimates which enable us to turn $L^2$ estimates on $s_k$ into pointwise ones. In case~3 we can even use pointwise estimates on the Green's operator.  At various steps we rely on facts about Green's operators which are essentially standard, but for which we were unable to find a clean reference which applies in the exact situations considered here. Accordingly, we have relegated proofs of these technical results to Appendix~\ref{Green}.

In order to fix our notation and conventions, we quickly recall the general form of the representation formula for systems which are not necessarily self adjoint. Suppose $D$ is an elliptic operator on sections of a vector bundle $E$ with a fibrewise metric, over a Riemannian manifold. Let $G(y,x) \in \Hom(E_x, E_y)$ be defined for all $x,y \in X$ with $x \neq y$, depending smoothly on $x$ and $y$. We say that \emph{$G$ is a fundamental solution for $D$} if it satisfies the following distributional equation: let $\sigma \in E_x$ and write $G(\cdot, x)(\sigma)$ for the section $y \mapsto G(y,x)(\sigma)$ of $E$; then
\[
D \left( G(\cdot, x)(\sigma)  \right) = \delta_x \sigma
\]
Explicitly, for any compactly supported section $s$ of $E$, 
\[
\int_X \left\langle G(y,x)(\sigma), D^* s(y) \right\rangle \dvol_y = \left\langle s(x), \sigma \right\rangle
\]
This is equivalent to the following representation formula: for any compactly supported section $s$ of $E$,
\begin{equation}
s(x)
	=
		\int_X
			G(y,x)^t(D^*s(y)) \dvol_y
\label{general-rep-formula}
\end{equation}
Notice in particular that a fundamental solution for $D$ gives a representation formula for $s$ in terms of $D^*s$. 

\subsection{Case 1}\label{case-1-section}

We assume that, after passing to a subsequence, $w(x_k) \gtrsim U_{k,\max}$ and so $x_k$ lives in the region of $X_k$ where $g_k$ is hyperbolic. We write $h_k$ for the (hyperbolic) metric on $X_k$ given by pulling back the hyperbolic metric via the branched cover $p \colon X_k\to M_k$. On $X_k\setminus \Sigma_k$ this metric is smooth, but incomplete. We will find a large embedded hyperbolic ball in $(X_k,h_k)$ centred at $x_k$ which we will then use to transfer everything over to hyperbolic space. The next two lemmas show that a suitably large ball can be found.

\begin{lemma}\label{distance-to-branch}
There exists $C>0$ such that for all large $k$,  $d_{h_k}(x_k,\Sigma_k) \geq \log (U_{k,\max}) - C$.
\end{lemma}
\begin{proof}
If $w(x_k) \geq \frac{1}{2}U_{k,\max}$ then, by definition of $w$ (in Lemma~\ref{weight-function}), 
\[
d_{h_k}(x_k,\Sigma_k) \geq \cosh^{-1}\left( \frac{1}{2}U_{k,\max}  \right)
\]
Now for all $\xi \geq 1$, $\cosh^{-1}(\xi) \geq \log(\xi)$ from which the lemma follows.

If $w(x_k) < \frac{1}{2}U_{k,\max}$ then $x_k$ is in the region where the hyperbolic coordinate system~\eqref{hyperbolic-metric-u} makes sense and $w =u$. The hyperbolic distance from $\Sigma_k$ is then given by
\[
d_{h_k}(x_k, \Sigma_k) = \cosh^{-1}(w(x_k))
\]
and the stated lower bound follows from the fact that $w(x_k) \gtrsim U_{k,\max}$.
\end{proof}

\begin{lemma}\label{hyperbolic-injectivity}
There exists  $C>0$ such that for all large $k$, $i_{h_k}(x_k) \geq \log(U_{k,\max}) - C$, where $i_{h_k}(x_k)$ denotes the injectivity radius of $h_k$ at $x_k$.
\end{lemma}
\begin{proof}
One possibility is that $i_{h_k}(x_k) = d_{h_k}(x_k,\Sigma_k)$; the exponential map of $h_k$ ceases to be an embedding because the geodesics reach the branch locus. In this case the result follows from the previous lemma. If $i_{h_k}(x_k) < d_{h_k}(x_k,\Sigma_k)$ then there is a geodesic loop $\gamma$, based at $x_k$ lying in $X_k \setminus \Sigma_k$ and of length $l(\gamma) = 2i_{h_k}(x_k)$. (Conjugate points cannot occur since $h_k$ is negatively curved.) Now the projection $p \colon X_k \to M_k$ is an isometric covering map; it follows that $p \circ \gamma$ is again a geodesic loop, and the lengths  satisfy $l(\gamma) \geq l(p\circ \gamma)$. Since $p \circ \gamma$ is non-constant, we have $ l(p\circ \gamma) \geq 2i(M_k)$ and so $i_{h_k}(x_k) \geq i(M_k)$. Since $U_{k,\max} = \cosh(\frac{1}{2}i(M_k))$, this implies that $i_{h_k}(x_k) \geq 2 \log(U_{k,\max})$ which is even stronger than the stated lower bound.
\end{proof}

Now let $0<\epsilon_0<1$ be some small number, to be chosen later. We set
\begin{equation}
R_k = (1-\epsilon_0) \log (U_{k,\max})
\label{Rk}
\end{equation}
By Lemma~\ref{hyperbolic-injectivity}, for $k$ sufficiently large, $R_k < i_{h_k}(x_k)$ and so the geodesic ball $B_{h_k}(x_k,R_k)$ is embedded in $X_k$. Let $\chi_k \colon \R_+ \to [0,1]$ be a sequence of cut-off functions with $\chi_k(r) = 1$ for $0 \leq r \leq R_k-1$, $\chi_k(r) = 0$ for $r \geq R_k$ and with $|\chi_k'| + |\chi_k''| \lesssim 1$. Write
\[
\tilde{s}_k = \chi_k(d_{h_k}(y,x_k)) s_k
\]
The cut-off tensor $\tilde{s}_k$ is supported in $B_{h_k}(x_k,R_k)$. We fix an isometric identification of this ball with a standard hyperbolic ball $B_h(p,R_k) \subset (\H^4,h)$. (Here, $h$ denotes the hyperbolic metric on $\H^4$.) We think of $\tilde{s}_k$ as a compactly supported tensor on $B_h(p,R_k)$ and thus, extending by zero, as a tensor on the whole of $\H^4$.

We next turn to the metrics $\tilde{g}_k$ of \eqref{metrics-converge}. We will show that over the ball $B_{h_k}(x_k,R_k)$ they are arbitrarily close to the hyperbolic metric $h_k$. By Lemma~\ref{distance-to-branch}, the geodesic ball $B_{h_k}(x_k,R_k)$ is contained in the region $\{ w \geq CU_{k,\max}^{\epsilon_0}\}$ for some constant $C$ (independent of $k$). Now it follows from the definition \eqref{interpolating-V} of $g_k$ that 
\[
\| g_k - h_k \|_{C^{m+2,\eta}(\{w \geq CU_{k,\max}^{\epsilon_0}\})} \lesssim U_{k,\max}^{-3\epsilon_0}
\]
Together with \eqref{metrics-converge} this gives
\begin{equation}
\| \tilde{g}_k - h_k \|_{C^{m+2,\eta}(B_{h_k}(x_k,R_k))} \to 0
\label{close-to-hyperbolic}
\end{equation}

Just as for $\tilde{s}_k$, we will consider $\tilde{g}_k$ as a metric on $B_h(p,R_k) \subset \H^4$. We then extend it to a new metric, still denoted $\tilde{g}_k$, on the whole of $\H^4$, which coincides with the original in $B_h(p,R_k)$, with the hyperbolic metric $h$ outside of $B(p,2R_k)$ and which interpolates between these two metrics in the intermediate region. By \eqref{close-to-hyperbolic} we can do this in such a way that, globally,
\begin{equation}
\| \tilde{g}_k - h \|_{C^{m+2,\eta}(\H^4)} \to 0
\label{global-close-to-hyperbolic}
\end{equation}
In an identical fashion we extend the original metrics $g_k$ from the ball $B_{h_k}(x_k,R_k) \cong B_h(p,R_k)$ to the whole of $\H^4$. We continue to denote this extension by $g_k$.

Since $\tilde{s}_k$ is compactly supported on $\H^4$, it has a Green's representation (see Appendix~\ref{Green}). Write $\tilde{H}_k(\cdot,p)$ for the fundamental solution of $L^*_{\tilde{g}_k}$, over $\H^4$, centred at $p$. (Recall the discussion leading up to~\eqref{general-rep-formula} which explains our conventions.) Note that by $L_{\tilde{g}_k}$ here we mean the linearised Einstein operator in Bianchi gauge relative to the extension of the metric $g_k$ to the whole of $\H^4$. Of course, over $B(p,R_k)$ this is the \emph{same} as the linearised operator back on $X_k$, in Bianchi gauge relative to $g_k$, hence we do not distinguish between them in the notation. The representation formula reads:
\begin{equation}
s_k(x_k)
=
\tilde{s}_k(p)
=
\int_{B_h(p,R_k)} \tilde{H}_k (q,p)^t 
	\left( L_{\tilde{g}_k}( \tilde{s}_k)(q)\right)\dvol_{\tilde{g}_k}(q)
\label{representation-case1}
\end{equation}
At points of $B_h(p,R_k)$ we have
\[
|L_{\tilde{g}_k} (\tilde{s}_k)| 
	\lesssim 
		|\chi_k| | L_{\tilde{g}_k} (s_k)| + |\chi_k'| |\nabla s_k| + |\chi_k''||s_k|
\]
The function $\chi_k$ is supported in $B_h(p,R_k)$ and $\chi_k'$ and $\chi''_k$ are supported in $R_k-1 \leq d_h(p,\cdot) \leq R_k$. So from \eqref{representation-case1} we have that
\begin{multline*}
|s_k(x_k)|_{g_k}
	\lesssim
		\left( 
		\int_{B_h(p,R_k)} \left| \tilde{H}_k(q,p)\right|_{\tilde{g}_k} \dvol_{\tilde{g}_k}(q)
		\right)
		\| L_{\tilde{g}_k} (s_k) \|_{C^0}\\
		+
		\left(  
		\int_{R_k-1 \leq d_{h}(p,q) \leq R_k}
		\left| \tilde{H}_k(q,p)\right|^2_{\tilde{g}_k} \dvol_{\tilde{g}_k}(q)
		\right)^{1/2}
		\left( 
		\| s_k \|_{L^2} + \| \nabla s_k \|_{L^2}
		\right)
\end{multline*}
To write this, in some places we used norms defined by the metric $g_k$ whilst in others we prefered $\tilde{g}_k$. This is allowed because they are equivalent uniformly in $k$ by assumption \eqref{metrics-converge}. Taking into account Lemma~\ref{preliminary-bounds} we have
\begin{multline}
|s_k(x_k)|_{g_k}
	\lesssim
		\left( 
		\int_{B_h(p,R_k)} \left| \tilde{H}_k(q,p)\right|_{\tilde{g}_k} \dvol_{\tilde{g}_k}(q)
		\right)
		\| L_{\tilde{g}_k} (s_k) \|_{C^0}\\
		+
		\left(  
		\int_{R_k-1 \leq d_{h}(p,q) \leq R_k}
		\left| \tilde{H}_k(q,p)\right|^2_{\tilde{g}_k} \dvol_{\tilde{g}_k}(q)
		\right)^{1/2}
		U^{\frac{5}{2}}_{k,\max}U_k^{\frac{3}{2} - \alpha} \| L_{\tilde{g}_k}(s_k)\|_{C^0_\alpha}
		\label{case1-estimate-in-pieces}
\end{multline}

We will estimate the right-hand side of \eqref{case1-estimate-in-pieces} using a weighted $L^2$ estimate on $\tilde{H}_k$. {We start with an optimal $L^2$-coercivity estimate for $L_{\tilde{g}_k}$. Since $\tilde{g}_k \to h$ in $C^2$, a combination of the arguments in the proof of Lemma~\ref{step1-invertibility} and in Delay \cite{delay} and Lee \cite{lee} (see for instance (4.3) of \cite{delay}) show that for all sufficiently large $k$ and for any symmetric bilinear form $s \in W^{2,2}(\H^4,\tilde{g}_k)$, 
\begin{equation} \label{L2-optimal}
\int_{\H^4} \left\langle L_{\tilde{g}_k}(s),s \right\rangle \dvol_{\tilde{g}_k}
	\geq
		\left( \frac{9}{8} -\delta  \right) \int_{\H^4} |s|^2_{\tilde{g}_k} \dvol_{\tilde{g}_k}
\end{equation}
holds. (In the special case of hyperbolic space itself one can take $\delta=0$; see Proposition~4.1 of \cite{delay}).}

We now give a technical Lemma which describes the effect of pulling a weight through $L_{\tilde{g}_k}$.

\begin{lemma}\label{pulling-rho-through-L}
Pick $\epsilon>0$. Let $\rho$ be a nowhere vanishing function in $C^2(\H^4, \tilde{g}_k)$ and  $s$ be a symmetric bilinear form in $W^{2,2}(\H^4,\tilde{g}_k)$. For all sufficiently large $k$ (depending on $\epsilon$), and any $\beta>0$,
\begin{multline*}
\int_{\H^4} \left\langle L_{\tilde{g}_k}(\rho^\beta s), \rho^\beta s \right\rangle \dvol_{\tilde{g}_k}
\leq 
(1+ \epsilon)\int_{\H^4} \rho^{2\beta} \left\langle L^*_{\tilde{g}_k} (s), s \right\rangle\dvol_{\tilde{g}_k}\\
+
\left(\frac{1}{2} +\epsilon\right) \beta^2\int_{\H^4} \rho^{2\beta}  \left| \frac{\diff \rho}{\rho} \right|_{\tilde{g}_k}^2 |s|_{\tilde{g}_k}^2 \dvol_{\tilde{g}_k}
+
\epsilon \int_{\H^4} \rho^{2\beta}|s|_{\tilde{g}_k}^2 \dvol_{\tilde{g}_k}
\end{multline*}
\end{lemma}
\begin{proof}
Let $\sigma \in W^{2,2}(\H^4, \tilde{g}_k)$. Recall the formula \eqref{linearised-Einstein-general} for the linearised Einstein operator in Bianchi gauge, which we write as:
\begin{equation}
L_{\tilde{g}_k}(\sigma) = \frac{1}{2}\nabla^*\nabla \sigma+ \div_{\tilde{g}_k}^*(D_k (\sigma))+ B_{g_k}(\tilde{g}_k) *\nabla (\sigma) + Z_k (\sigma)
\label{schematic-formula-L}
\end{equation}
where 
\[
D_k (\sigma) 
	= 
		(\div_{g_k} - \div_{\tilde{g}_k})(\sigma) 
		+ 
		\frac{1}{2} \diff \left( \tr_{g_k} \sigma - \tr_{\tilde{g}_k} \sigma \right)
\]
and $Z_k$ is zeroth order. Integrating by parts we have, for any $\sigma \in W^{2,2}(\H^4, \tilde{g}_k)$,
\begin{multline*}
\int_{\H^4} 
	\left\langle L_{\tilde{g}_k} (\sigma), \sigma  \right\rangle 
	\dvol_{\tilde{g}_k}
		=
			\frac{1}{2}\int_{\H^4}
			\left\langle \nabla^*\nabla \sigma, \sigma   \right\rangle
			\dvol_{\tilde{g}_k}
			+
			\int_{\H^4}
			\left\langle Z_k (\sigma), \sigma \right\rangle \dvol_{\tilde{g}_k}\\
			+
			\int_{\H^4} \Big(
			\left\langle D_k(\sigma), \div_{\tilde{g}_k}(\sigma) \right\rangle
			+
			\left\langle B_{g_k}(\tilde{g}_k) * \nabla \sigma, \sigma \right\rangle
			\Big)\dvol_{\tilde{g}_k}
\end{multline*}
Since $\|\tilde{g}_k- g_k\|_{C^2} =o(1)$, the terms involving $D_k$ and $B_{g_k}(\tilde{g}_k)$ can be bounded by arbitrarily small amounts of the $W^{1,2}$-norm of $\sigma$. (This is identical to the argument used to derive \eqref{I2}, \eqref{I3} and~\eqref{I4}.) {Choosing $\sigma = \rho^{ \beta} s$ and integrating by parts (as in the arguments leading up to (7.14) in Lee \cite{lee}) shows that:}
\begin{multline}
\int_{\H^4} 
\left\langle L_{\tilde{g}_k} (\rho^\beta s), \rho^\beta s  \right\rangle \dvol_{\tilde{g}_k}
	\leq\\
		\left(\frac{1}{2} + \epsilon \right)
		\int_{\H^4}\rho^{2\beta} 
		\left\langle \nabla^*\nabla s, s \right\rangle \dvol_{\tilde{g}_k}		 		+
		 \int_{\H^4} \rho^{2\beta} \left\langle Z_ks,s \right\rangle
		 \dvol_{\tilde{g}_k}\\
		+
		\left(\frac{1}{2} + \epsilon\right) \beta^2
		\int_{\H^4} \rho^{2\beta} \left|\frac{\diff \rho}{\rho} \right|_{\tilde{g}_k}^2 |s|_{\tilde{g}_k}^2
		 \dvol_{\tilde{g}_k}
		 +
		 \epsilon \int_{\H^4} \rho^{2\beta}|s|_{\tilde{g}_k}^2 \dvol_{\tilde{g}_k}
		 \label{bound-before-rho-pulled}
\end{multline}

We will bound the first two terms on the right-hand side. From \eqref{schematic-formula-L} applied to $\sigma = \rho^{2 \beta}s$ and integrated against $s$ we have, {since $\tilde{g}_k \to h$ in $C^2$}

\begin{multline}
\frac{1}{2}\int_{\H^4} 
\rho^{2\beta} \left\langle \nabla^*\nabla s, s \right\rangle 
\dvol_{\tilde{g}_k}
	+
\int_{\H^4} \rho^{2\beta} \left\langle Z_k s, s \right\rangle
\dvol_{\tilde{g}_k}
		\leq\\
			\int_{\H^4}\rho^{2\beta}
			\left\langle L^*_{\tilde{g}_k} (s), s \right\rangle \dvol_{\tilde{g}_k}
			+
			\epsilon \int_{\H^4} \rho^{2\beta} 
			\left( 
			|\nabla s|_{\tilde{g}_k}^2 + |s|_{\tilde{g}_k}^2 
			+	
			\beta^2 \left|\frac{\diff \rho}{\rho} \right|_{\tilde{g}_k}^2 |s|_{\tilde{g}_k}^2
			\right)
			 \dvol_{\tilde{g}_k}
			 \label{half-way-there}
\end{multline}
Now, integrating by parts 
\[  \int_{\H^4} \textrm{div} \Big(\rho^{2 \beta} \nabla (|s|_{\tilde{g}_k}^2) \Big)\dvol_{\tilde{g}_k} = 0 \]
 we deduce that
\begin{align*}
\int_{\H^4} \rho^{2\beta} |\nabla s|_{\tilde{g}_k}^2 \dvol_{\tilde{g}_k}
	&\leq
		2\int_{\H^4} \rho^{2\beta} 
		\left\langle \nabla^*\nabla s, s \right\rangle \dvol_{\tilde{g}_k}
		+
		4\beta^2 \int_{\H^4} \rho^{2\beta}
		\left|\frac{\diff \rho}{\rho}\right|_{\tilde{g}_k}^2|s|_{\tilde{g}_k}^2 \dvol_{\tilde{g}_k}
\end{align*}
and so \eqref{half-way-there} implies that
\begin{multline}
\left(\frac{1}{2} - 2\epsilon\right)\int_{\H^4} 
\rho^{2\beta} \left\langle \nabla^*\nabla s, s \right\rangle 
\dvol_{\tilde{g}_k}
	+
\int_{\H^4} \rho^{2\beta} \left\langle Z_k (s), s \right\rangle
\dvol_{\tilde{g}_k}
		\leq\\
			\int_{\H^4}\rho^{2\beta}
			\left\langle L^*_{\tilde{g}_k} (s), s \right\rangle \dvol_{\tilde{g}_k}
			+
			\epsilon \int_{\H^4} \rho^{2\beta} 
			\left( 
			|s|_{\tilde{g}_k}^2 
			+	
			5\beta^2 \left|\frac{\diff \rho}{\rho} \right|_{\tilde{g}_k}^2 |s|_{\tilde{g}_k}^2
			\right)
			 \dvol_{\tilde{g}_k}
			 \label{another-messy-inequality}
\end{multline}
Together with \eqref{bound-before-rho-pulled} we conclude that
\begin{multline}
\int_{\H^4} \left
\langle L_{\tilde{g}_k} (\rho^\beta s) , \rho^\beta s \right\rangle \dvol_{\tilde{g}_k}
	\leq\\
		\int_{\H^4} \rho^{2\beta} 
		\left\langle  L^*_{\tilde{g}_k} (s), s\right\rangle \dvol_{\tilde{g}_k}
		+
		\left( \frac{1}{2} + 5 \epsilon \right)
		\beta^2\int_{\H^4} \rho^{2\beta}
		\left|\frac{\diff \rho}{\rho}\right|_{\tilde{g}_k}^2 |s|_{\tilde{g}_k}^2 \dvol_{\tilde{g}_k}
		+
		2 \epsilon\int_{\H^4} \rho^{2\beta}|s|_{\tilde{g}_k}^2 \dvol_{\tilde{g}_k}\\
		+
		3 \epsilon \int_{\H^4}	\rho^{2\beta} 
		\left\langle \nabla^*\nabla s, s \right\rangle\dvol_{\tilde{g}_k}
		\label{yet-more-mess}
\end{multline}
{The last term in \eqref{yet-more-mess} is controlled using \eqref{another-messy-inequality} and the fact that $Z_k$ is bounded uniformly in $k$, and this proves the result (after choosing a different~$\epsilon$).}
\end{proof}

We now reach the technical crux: the weighted $L^2$ estimate for $\tilde{H}_k$. The following result can be interpreted as saying that $\tilde{H}_k$ has the same decay at infinity as the Green's operator for hyperbolic space.

\begin{proposition}\label{weighted-L2-tildeH}
Given $0 < \epsilon < 3$, there exists a constant $C$ (depending only on $\epsilon$, but not on $k$) such that
\[
\int_{\H^4 \setminus B_{\tilde{g}_k}(p,1)}
	e^{(3- \epsilon) d_{\tilde{g}_k}(p,q)}
		\left| \tilde{H}_k(q,p) \right|^2_{\tilde{g}_k} \dvol_{\tilde{g}_k}(q)
\leq 
	C
\]
\end{proposition}

\begin{proof}
Let $\sigma \in S^2T^*_p\H^4$ with $|\sigma|=1$, and define
\[
F(q) = \tilde{H}_k(q,p)(\sigma)
\]
$F$ is the fundamental solution with $L^*_{\tilde{g}_k} F = \delta_p\cdot \sigma$ (in the distributional sense). Since $\sigma$ is arbitrary, it suffices to prove the stated inequality with $|\tilde{H}_k(q,p)|_{\tilde{g}_k}^2$ replaced by $|F(q)|_{\tilde{g}_k}^2$.

Let $\psi \colon [0, \infty) \to \R$ be a smooth bump function with $\psi \equiv 0$ on $[0, 1/2]$ and $\psi \equiv 1$ on $[1, \infty)$. Let $\Psi_k(q) = \psi(d_{\tilde{g}_k}(p,q))$ and $\tilde{F} = \Psi_k F$. Then $\tilde{F}$ is smooth on all of $\H^4$. 

Next, pick $M>0$ and let $\chi \colon [0, \infty) \to \R$ be a smooth function with $\chi \equiv 2$ in $[0,1]$, $\chi(r) = r$ for $r\in [3,M]$ and $\chi \equiv M+1$ in $[M+2,\infty)$. We choose $\chi$ so that $|\chi'(r)|\leq 1$ for all $r$. We now let 
\[
\rho(q) = e^{\chi (d_{\tilde{g}_k}(p,q))}
\]
This function is $C^2$ globally and so can be used in Lemma~\ref{pulling-rho-through-L}. 

Let $\delta >0$ be small
and choose $0 < \beta < 3/2$. By ~\eqref{L2-optimal} and Lemma~\ref{pulling-rho-through-L}, we have that for all sufficiently large $k$, 
\begin{align}
\left( \frac{9}{8} - \delta \right)
\int_{\H^4} \rho^{2\beta} |\tilde{F}|_{\tilde{g}_k}^2 \dvol_{\tilde{g}_k}
	&\leq
		\left( 1 + \delta \right) \int_{\H^4}\rho^{2\beta}\left\langle  
		L^*_{\tilde{g}_k} (\tilde F), \tilde{F} \right\rangle \dvol_{\tilde{g}_k}
		\nonumber\\
	&
		\quad\quad\quad\quad
		+\left( \frac{1}{2} + \delta \right)\beta^2 \int_{\H^4} \rho^{2\beta}
		\left | \frac{\diff \rho}{\rho} \right|_{\tilde{g}_k}^2 |\tilde{F}|_{\tilde{g}_k}^2 \dvol_{\tilde{g}_k}	
		\nonumber\\
	&
		\quad\quad\quad\quad\quad\quad\quad\quad
		+ \delta \int_{\H^4} \rho^{2\beta} |\tilde{F}|_{\tilde{g}_k}^2 \dvol_{\tilde{g}_k}
		\label{nearly-there}
\end{align}

Since $F$ is a fundamental solution, $L^*_{\tilde{g}_k}\tilde F = L^*_{\tilde{g}_k}(\Psi_k F)$ is supported in the annulus $A = B_{\tilde{g}_k}(p,1)\setminus B_{\tilde{g}_k}(p,1/2)$. So in the first term on the right-hand side of \eqref{nearly-there}, the integrand is supported inside $A$. Moreover, since $\tilde{g}_k$ has uniformly bounded geometry, $\tilde{H}_k(q,p)$ is uniformly bounded in $C^2$ over $A$ (see the remark following Proposition~\ref{propGreen} in the Appendix). Also, since $\tilde{g}_k \to h$ in $C^{m+2, \eta}$, the functions $\Psi_k$ are also uniformly bounded in $C^2$. {It follows that there is a constant $C$, independent of $k$, such that}
\begin{equation*}
\int_{\H^4} \rho^{2\beta} 
	\left\langle L^*_{\tilde{g}_k} (\tilde F), \tilde F \right\rangle \dvol_{\tilde{g}_k}
		\leq 
			C
\label{constant-term}
\end{equation*}

Now, since we were careful to select $|\chi'(r)|\leq 1$, we have $|\diff \rho| \leq \rho$. Using this, and rearranging \eqref{nearly-there} we obtain
\[
\left( \frac{9}{8} - \frac{1}{2} \beta^2 - \delta \beta^2  - 2\delta\right)
\int_{\H^4} \rho^{2\beta} |\tilde{F}|_{\tilde{g}_k}^2 \dvol_{\tilde{g}_k}
\leq 
C
\]
Since $0 < \beta < 3/2$, we can find $\delta>0$ such that the coefficient on the left-hand side here is positive. {Given $\epsilon >0$, we let $\beta = (3-\epsilon)/2$. Taking $M \to \infty$ completes the proof since $\tilde F= F$ at points of $\H^4\setminus B_{\tilde{g}_k}(p,1)$. }
\end{proof}

We now convert this weighted estimate into unweighted integral estimates on $\tilde{H}_k$.

\begin{lemma}\label{integral-bounds-tildeH}
Let $0< \epsilon < 3$ and $0<\epsilon_0<1$ and let $R_k = (1-\epsilon_0)\log(U_{k,\max})$ as in \eqref{Rk}. Then
\begin{align}
\int_{R_k - 1 \leq d_h(p,q) \leq R_k}
	|\tilde{H}_k(q,p)|_{\tilde{g}_k}^2 \dvol_{\tilde{g}_k}(q)
	&\lesssim
		U_{k,\max}^{(\epsilon - 3)(1-\epsilon_0)}
	\label{L2-big-annulus-bound}
		\\
\int_{B_h(p,R_k)} |\tilde{H}_k(q,p)|_{\tilde{g}_k} \dvol_{\tilde{g}_k}(q)
	&\lesssim
		U^{\epsilon(1-\epsilon_0)}_{k,\max}
	\label{L1-big-ball-bound}
\end{align}
\end{lemma}
\begin{proof}
To prove \eqref{L2-big-annulus-bound}, note that since $\tilde{g}_k \to h$, if $q$ has $d_h(p,q) \leq R_k$ then for all sufficiently large $k$, $d_{\tilde{g}_k}(p,q) \leq R_k +1$. It follows that 
\begin{align*}
\int_{R_k - 1 \leq d_h(p,q) \leq R_k} |\tilde{H}_k(q,p)|_{\tilde{g}_k}^2 \dvol_{\tilde{g}_k}(q)
	&\lesssim
		e^{(\epsilon - 3)R_k}	\int_{\H^4 \setminus B_h(p,1)} 
		e^{(3-\epsilon)d_{\tilde{g}_k}(p,q)}|\tilde{H}_k(q,p)|_{\tilde{g}_k}^2 \dvol_{\tilde{g}_k}(q)\\
	&\lesssim
		U_{k,\max}^{(\epsilon - 3)(1-\epsilon_0)}
\end{align*}

To prove \eqref{L1-big-ball-bound}, we start with a fact about the Green's operator, proved in the Appendix as Proposition~\ref{propGreen}, that $\tilde{H}_k(q,p)$ is uniformly integrable on $B_h(p,1)$. So we have
\begin{align*}
\int_{B_h(p,R_k)} |\tilde{H}(q,p)|_{\tilde{g}_k} \dvol_{\tilde{g}_k}
	& \leq
		C 
		+ 
		\int_{B_h(p,R_k) \setminus B_{\tilde{g}_k}(p,1)} 
		|\tilde{H}_k(q,p)|_{\tilde{g}_k} \dvol_{\tilde{g}_k}\\
	&\leq
		C + 
		\left(  
		\int_{B_{\tilde{g}_k}(p,R_k+1)} e^{(\epsilon - 3)d_{\tilde{g}_k}(p,q)} \dvol_{\tilde{g}_k}(q)
		\right)^{1/2}
\end{align*}
In the second line, we have used Cauchy--Schwarz and the weighted $L^2$-estimate of Proposition~\ref{weighted-L2-tildeH}. We have also replaced $B_h(p,R_k)$ with the larger ball $B_{\tilde{g}_k}(p,R_k+1)$. This second ball genuinely is larger, since $\tilde{g}_k \to h$. 

We estimate this last integral using geodesic normal coordinates $(r,\theta)$ for $\tilde{g}_k$ centred at $p$. Since $\tilde{g}_k \to h$ in $C^2$, we have the uniform lower bound $\Ric(\tilde{g}_k) \geq - (3+\epsilon)$ for all large $k$. By Bishop--Gromov comparison, $\dvol_{\tilde{g}_k} \lesssim \sinh((3+\epsilon)r) \diff r \wedge \diff \theta$. From this and the definition of $R_k$ it follows that
\[
\int_{B_{\tilde{g}_k}(p,R_k)} e^{(\epsilon - 3)d_{\tilde{g}_k}(p,q)} \dvol_{\tilde{g}_k}(q)
\lesssim
U^{2\epsilon(1-\epsilon_0)}_{k,\max}
\]
Taking the square root completes the proof
\end{proof}

We are finally ready to prove the result which will contradict \eqref{bound-to-be-violated} {(with $p<1/8$)} showing that this case cannot actually occur after all.

\begin{proposition}\label{case1-contradiction}
Let $0<\alpha < 1/4$ and choose the gluing parameter $U_k$ so that
\begin{equation}
U_k^{-\frac{3}{2}} U_{k,\max}^{\frac{5}{4} +\alpha} \to 0
\label{gluing-parameter-choice}
\end{equation}
Then for all sufficiently large $k$, 
\[
\| s_k\|_{C^0_\alpha} \leq U_k^{3 - \alpha - \frac{1}{8}}\| L_{\tilde{g}_k}(s_k)\|_{C^0_\alpha}
\]
\end{proposition}
\begin{proof}
Recall that $x_k$ is the point at which $|w^\alpha s_k|_{g_k} = \| s_k \|_{C^0_\alpha}$. This (and the assumption that we are in Case~1) implies that $\|s_k\|_{C^0_\alpha} \leq U_{k,\max}^\alpha |s_k(x_k)|$. Now we apply \eqref{case1-estimate-in-pieces} together with Lemma~\ref{integral-bounds-tildeH}: for any $0< \epsilon <3$ and any $0<\epsilon_0 <1$,
\[
\| s_k \|_{C^0_\alpha}
	\lesssim
		\left(
		U^{\alpha + \epsilon(1-\epsilon_0)}_{k,\max} U_k^{-\alpha} 
		+ 
		U_{k,\max}^{\alpha +\frac{5}{2}+ \frac{1}{2}(\epsilon - 3)(1-\epsilon_0)}
		U_k^{\frac{3}{2}-\alpha}
		\right)
		\| L_{\tilde{g_k}}(s_k)\|_{C^0_\alpha}
\]
(We have used the fact that $L_{\tilde{g}_k}(s_k)$ is supported in the region $S_k$ which implies that $\|L_{\tilde{g}_k}(s_k)\|_{C^0} \lesssim U_k^{-\alpha} \| L_{\tilde{g}_k}(s_k)\|_{C^0_\alpha}$.) We can write this coefficient as
\[
U_k^{3 -\alpha} 
\left( 
U_k^{-3}U_{k,\max}^{\alpha + O(\epsilon,\epsilon_0)}
+
U_k^{-\frac{3}{2}} U_{k,\max}^{\alpha + 1 + O(\epsilon, \epsilon_0)}
\right)
\]
By taking $\epsilon$ and $\epsilon_0$ sufficiently small and using the hypotheses on $\alpha$ and the choice of $U_k$, we get
\[
\| s_k \|_{C^0_\alpha}
	\lesssim
		U_k^{3-\alpha} U_{k,\max}^{-\frac{1}{8}} \left( U_k^{-\frac{3}{2} }U_{k,\max}^{\alpha + \frac{5}{4}}\right) \| L_{\tilde{g}_k}(s_k)\|_{C^0_\alpha}	
	\lesssim 
		U_k^{3-\alpha - \frac{1}{8}} \| L_{\tilde{g}_k}(s_k)\|_{C^0_\alpha}	
\qedhere\]
\end{proof}
\subsection{Case 2}\label{case-2-section}

We assume that, after passing to a subsequence,  $w(x_k) \leq C$ for some constant $C$ independent of $k$. By definition of $g_k$, this means that the geodesic distance from $x_k$ to $\Sigma_k$ is uniformly bounded and so less than the normal injectivity radius (for all large $k$). We denote by $\Sigma_k'$ the nearest component of $\Sigma_k$ to $x_k$. Just as in \S\ref{approximate-solution}, we identify a tubular neighbourhood of $\Sigma_k'$ with
\[
\frac{[u_a, U_{k,\max}) \times S^1 \times \Sigma_k' }{ \sim}
\]
The quotient by the relation $\sim$ denotes that we have collapsed the $S^1$ factor over $\{u_a\} \times \Sigma_k$ to produce a smooth manifold without boundary. We use $u \in [u_a, U_{k,\max})$ for the corresponding coordinate function in the radial direction, as in \S\ref{approximate-solution}, which behaves at large distances as the exponential of the distance to $\Sigma_k$. Here the minimal value $u_a$ is the constant defined in the course of Lemma~\ref{cone-angle}. We recall that it depends only on the degree $l$ of the cover, and not on $k$. We then transfer everything to the non-compact manifold
\[
Y_k =  \frac{[u_a,\infty) \times S^1 \times \Sigma_k'}{\sim}
\]
The approximate Einstein metric $g_k$ restricts from $X_k$ to the region $u \leq U_{k,\max}$ of $Y_k$; it is hyperbolic for $u \geq U_{k,\max}/2$ and so extends directly, remaining hyperbolic, to the rest of $Y_k$. We continue to denote this extension by $g_k$. The metric $\tilde{g}_k$, satisfying \eqref{metrics-converge} restricts to the region $u \leq U_{k,\max}/2$ of $Y_k$; we then extend it to the whole of $Y_k$ by interpolating with $g_k$ over the region $U_{k,\max}/2 \leq u \leq U_{k,\max}$. This gives a metric on the whole of $Y_k$ which we continue to denote by $\tilde{g}_k$. We remark that we still have the analogue of \eqref{metrics-converge}, namely
\begin{equation}\label{localised-metrics-converge}
\| \tilde{g}_k - g_k \|_{C^{m+2,\eta}_\alpha(Y_k)} \to 0
\end{equation}

The strategy is the same as for Case~1: we prove weighted $L^2$ estimates on the Green's operator which, together with the global $W^{1,2}$-estimates of Lemma~\ref{preliminary-bounds} lead to a contradiction with \eqref{bound-to-be-violated}. This time, however, we use the function $u$ as a weight. This choice is motivated by the fact that for the asymptotically hyperbolic model of \S\ref{approximate-solution}, $u^{-1}$ is a boundary defining function, and such functions are the appropriate weight to use in that context.

For any pair $x,y$ of disjoint points of $Y_k$, denote by $\tilde{G}_k(y,x)$ the fundamental solution of $L_{\tilde{g}_k}^*$ in $Y_k$, centred at $x$. (See Appendix~\ref{Green} for the construction of $\tilde{G}_k$.) We have the following weighted $L^2$ estimate on $\tilde{G}_k$.

\begin{proposition}\label{weighted-L2-case2}
Given $0 < \epsilon <3$, there exists a constant $C$ (depending only on $\epsilon$, but not on $k$) such that for any $x \in Y_k$ and for all large $k$,
\[
\int_{Y_k \setminus B_{g_k}(x,1)}
u(y)^{3-\epsilon}
| \tilde{G}_k(y,x)|^2_{\tilde{g}_k} \dvol_{\tilde{g}_k}
	\leq
		C u(x)^{3-\epsilon}
\]
\end{proposition}
\begin{proof}
The proof follows the same lines as that of Proposition~\ref{weighted-L2-tildeH} with minor modifications. Accordingly we pass more quickly over the steps this time. The main difference is that the analogue of~\eqref{L2-optimal} now only holds asymptotically here. 
{We first claim that for any $\delta>0$ there exists $U_0$ such that for any smooth symmetric bilinear form $s$ compactly supported in $\{ u \geq U_0\}$ and any $k $ large enough we have
\begin{equation}
\int_{Y_k}
\left\langle  L^*_{\tilde{g}_k}(s), s\right\rangle \dvol_{\tilde{g}_k}
	=
\int_{Y_k} 
\left\langle  L_{\tilde{g}_k}(s), s\right\rangle \dvol_{\tilde{g}_k}
	\geq
		\left( \frac{9}{8} - \frac{\delta}{2} \right)
		\int_{Y_k} |s|_{\tilde{g}_k}^2 \dvol_{\tilde{g}_k}
\label{asymptotic-L2-optimal}
\end{equation}
To prove this, note that $(Y_k,\tilde{g}_k)$ is an asymptotically hyperbolic manifold with boundary defining function~$u^{-1}$ (it is even exactly hyperbolic at large distances). By definition of $g_k$ and by \eqref{localised-metrics-converge} we have  in particular that for any $U_0 \ge 1$
\begin{equation} \label{metric-proximity}
  \Vert \tilde{g}_k - g_{hyp} \Vert_{C^{m+2, \eta}(\{ u \ge U_0\})} \le C U_0^{-\alpha}
  \end{equation}
for some positive constant $C$ which is independent of $k$, where $g_{hyp}$ denotes the hyperbolic metric in the subset $\{ u \ge U_0 \} \subset Y_k$ given by
\[
g_{hyp} =  \frac{\diff u^2}{u^2-1} + l^2(u^2-1)\diff \theta^2 + u^2 h_{\Sigma_k'},
\]
with $h_{\Sigma_k'}$ the hyperbolic metric on $\Sigma_k'$. The rate of decay of $\tilde{g}_k$ to the hyperbolic metric therefore only depends on the radial coordinate $u$ and so is uniform in $k$. With \eqref{metric-proximity}, the proof of \eqref{asymptotic-L2-optimal} now mimicks the one of \eqref{L2-optimal} (see also Lemma $7.13$ in \cite{lee}), and \eqref{asymptotic-L2-optimal} holds true uniformly in $k$ provided $U_0$ is large enough.}

Fix $\delta>0$ small and an associated $U_0$ as in \eqref{asymptotic-L2-optimal}. Let $\eta \colon [0, \infty) \to [0,\infty)$ be a smooth function with $\eta(u) = 0$ when $u \leq U_0$ and $\eta(u) = 1$ for $u \geq 2U_0$. We will use $\eta$ to support the Green's function in the region $\{u \geq U_0\}$. Meanwhile we will use a second cut-off function $\chi$ to cut-off at large values of $u$, as in the proof of Proposition~\ref{weighted-L2-tildeH}. For this, pick $M \gg 1$ and let $\chi \colon [0, \infty) \to [0, \infty)$ be smooth with $\chi(u) = U_0 + 1$ when $u \leq U_0$, $\chi(u) = u$ for $U_0 +1 \leq u \leq M-1$ and $\chi(u) = M$ when $u \geq M+1$. We choose $\chi$ so that $|\chi'| \leq 1$. 

Let $x \in Y_k$ and $\sigma \in S^2T_x^*Y_k$ with $|\sigma|_{\tilde{g}_k}=1$ and as before, put
\[
F(y) = \tilde{G}_k(y,x)(\sigma)
\]
Let $\psi$ be a smooth cut-off function centred at $x$, with $\psi \equiv 0$ in $B_{\tilde{g}_k}(x,1/2)$ and $\psi \equiv 1$ in $Y_k \setminus B_{\tilde{g}_k}(x,1)$. We put
\[
\tilde{F}(y) = \psi(y) \eta(u(y)) F(y)
\]
$L^*_{\tilde{g}_k}(\tilde{F})$ is supported in the union of the annulus $B_{\tilde{g}_k}(x,1) \setminus B_{\tilde{g}_k}(x,1/2)$ and the region $\{ U_0 \leq u\leq 2U_0\}$. Arguing as in the derivation of \eqref{constant-term} we get that
\begin{equation}
\int_{Y_k} \chi(u(y))^{3-\epsilon} | L^*_{\tilde{g}_k}\tilde{F}(y)|^2_{\tilde{g}_k}\dvol_{\tilde{g}_k} \leq C u(x)^{3-\epsilon}
\label{case2-constant-term}
\end{equation}
where $C$ does not depend on $\epsilon, M$ or $k$ but does depend on $U_0$. Two points in the proof of \eqref{case2-constant-term} are different from the previous discussion. On the one hand, to control the contribution of the integrand supported in $B_{\tilde{g}_k}(x,1) \setminus B_{\tilde{g}_k}(x,1/2)$ we must first bound uniformly the volume of the unit ball:
\[
\vol(B_{\tilde{g}_k}(x,1)) \leq C
\]
This follows from the Bishop-Gromov inequality and the fact that the Ricci curvature of $\tilde{g}_k$ is uniformly bounded below. Then, we use that for large $k$ we have $| \diff u|_{\tilde{g}_k}\leq 2 V(u)^{1/2} \leq 2u$, where $V$ is defined in \eqref{interpolating-V}. It follows that there is a constant $C$ such that for any $y \in B_{\tilde{g}_k}(x,1)$, $u(y) \leq C u(x)$.

On the other hand, the other possible support of the integrand in \eqref{case2-constant-term} is the region $\{ U_0 \leq u \leq 2 U_0\}$. For such $y$, we also have $u(y) \leq C u(x)$ where $C$ now depends on $U_0$ (this is because $u(x) \geq u_a >0$), and therefore $\chi(u(y))^{3-\epsilon} \leq Cu(x)^{3-\epsilon}$. Together with the global uniform $L^2$ control on $\tilde{G}_k(y,x)$ and its covariant derivative given in Proposition~\ref{propGreen} of the Appendix, this proves \eqref{case2-constant-term}.

We now copy precisely the steps in the proof of Proposition~\ref{weighted-L2-tildeH} to obtain, for any $0 < \beta < 3/2$,
\[
\int_{Y_k} \chi(u(y))^{2\beta} |\tilde{F}(y)|^2_{\tilde{g}_k} \dvol_{\tilde{g}_k}
\leq
C u(x)^{2\beta}
\]
Here $C$ depends only on $\beta$ and $U_0$, but not on $M$, $k$ or $x$. Now letting $M \to \infty$ 
gives, by definition of $\eta$ and $\psi$: 
\[
\int_{\{u \geq 2U_0\}\setminus B_{\tilde{g}_k}(x,1)}
	u(y)^{2\beta}|F(y)|^2_{\tilde{g}_k}\dvol_{\tilde{g}_k} 
		\leq
			Cu(x)^{2\beta}
\]
Finally, the integral over the region $\{ u \leq 2U_0\}\setminus B_{\tilde{g}_k}(x,1)$ is independently estimated by the global $L^2$ bound on $\tilde{G}_k(y,x)$ given in Proposition~\ref{propGreen}. This completes the proof.
\end{proof}

Note that this time, in Proposition \eqref{weighted-L2-case2}, the weight is just a function of $u$ and is not normalized relative to $x$.

With this weighted $L^2$-estimate in hand, we prove the following bound on $s_k(x_k)$ which gives a contradiction with~\eqref{bound-to-be-violated} (again with $p = 1/8$). (For consistency, we choose $U_k$ to satisfy the same constraint \eqref{gluing-parameter-choice}. If we only cared about Case~2, we could have used a weaker constraint on $U_k$.)

\begin{proposition}
Let $1/8 < \alpha < 1/4$ and choose the gluing parameter $U_k$ so that
\[
U_k^{-\frac{3}{2}} U_{k, \max}^{\frac{5}{4} + \alpha} \to 0
\]
Then for all sufficiently large $k$, 
\[
\| s_k \|_{C^0_\alpha} 
	\leq 
		U_k^{3 - \alpha - \frac{1}{8}}
		\| L_{\tilde{g}_k} (s_k) \|_{C^0_\alpha}
\]
\end{proposition}
\begin{proof}
We first transport $s_k$ from $X_k$ to the model space $Y_k$. Let $\eta_k \colon [0,\infty) \to [0, \infty)$ be a smooth cut-off function with $\eta_k(u) =1$ for $u \leq U_{k,\max}/4$, $\eta_k(u) = 0$ when $u \geq U_{k,\max}/2$ and 
\[
| \eta'_k| + U_{k,\max} |\eta_k''| \lesssim \frac{1}{U_{k,max}}
\]
For such a choice of $\eta_k$ we have, in particular, that
\begin{equation}
| \nabla \eta_k (u(\cdot))|_{\tilde{g}_k}
+ 
| \Delta \eta_k(u(\cdot))|_{\tilde{g}_k}
\lesssim 1
\label{etak-second}
\end{equation}
Let $\tilde{s}_k = \eta_k s_k$. This defines a symmetric tensor field supported in the region $u \leq U_{k,\max}$ and we extend it by zero to a tensor on the whole of $Y_k$. We now take the Green's representation formula for $\tilde{s}_k$:
\[
s_k(x_k) 
	=
		\tilde{s}_k(x_k)
	=
		\int_{Y_k} \tilde{G}_k(y, x_k)^t \left(L_{\tilde{g}_k} \big(\tilde{s}_k\big) (y)\right)
		\dvol_{\tilde{g}_k}(y)
\]		
Now $L_{\tilde{g}_k}(s_k)$ is by assumption supported in the region $S_k = \{ U_k/2 \leq u \leq U_k\}$, where $\eta_k = 1$. It follows that 
\begin{multline}
| s_k(x_k)|_{g_k} 
	\lesssim
		\left(
		\int_{S_k} |\tilde{G}_k(y,x_k)|_{\tilde{g}_k} \dvol_{\tilde{g}_k} 
		\right)
		U^{-\alpha}_k \| L_{\tilde{g}_k} (s_k) \|_{C^{0}_\alpha} \\
		+
		\left( 
		\int_{\frac{U_{k,\max}}{4} \leq u \leq \frac{U_{k,\max}}{2}}
		|\tilde{G}_k(y,x_k)|^2_{\tilde{g}_k}
		\right)^{1/2}
		\left( \| s_k \|_{L^2} + \| \nabla s_k\|_{L^2} \right)
\label{rep-bound-case2}
\end{multline}
(we used here the uniform bounds \eqref{etak-second} on the derivatives of $\eta_k$.)

We make use of Proposition~\ref{weighted-L2-case2}. Since we are assuming here that $u(x_k) \leq C$ is uniformly bounded, we can replace the bound of this result by a uniform constant, and we in particular have that $x_k \not \in S_k$.

By Cauchy--Schwarz,
\begin{align*}
\int_{S_k} | \tilde{G}_k(y,x_k)|_{\tilde{g}_k} \dvol_{\tilde{g}_k}
	&\leq
		\left(  
		\int_{S_k}u(y)^{\epsilon - 3}\dvol_{\tilde{g}_k}(y)
		\right)^{1/2}
		\left(  
		\int_{S_k}u(y)^{3-\epsilon}|\tilde{G}_k(y,x_k)|^2_{\tilde{g}_k}
		\dvol_{\tilde{g}_k}(y)
		\right)^{1/2}\\
	&\lesssim
		U_k^{\frac{1}{2}(\epsilon - 3)}\vol(S_k)^{1/2}
\end{align*}
Now $\vol(S_k) \leq U_k^3 \vol(\Sigma_k)$ where $\vol(\Sigma_k)$ is the hyperbolic volume of the branch locus $\Sigma_k$. By part 3 of Proposition~\ref{hyperbolic-sequence}, we have $\vol(\Sigma_k) \lesssim U_{k,\max}^5$ (as discussed before~\eqref{volume-branch}) and from here we have that
\begin{equation}
\int_{S_k} | \tilde{G}_k(y,x_k)|_{\tilde{g}_k} \dvol_{\tilde{g}_k}
	\lesssim
		U_k^{\frac{\epsilon}{2}}U_{k,\max}^{\frac{5}{2}}
	=
		o(U_k^{3 - \frac{1}{8}})
\label{L1-over-Sk}
\end{equation}
(where we have used the condition on $U_k$ in the hypotheses, with $\epsilon>0$ chosen sufficiently small). 

This deals with the first term in \eqref{rep-bound-case2}. For the second term we use Proposition~\ref{weighted-L2-case2} to write
\[
\int_{\frac{U_{k,\max}}{4} \leq u \leq \frac{U_{k,\max}}{2}}
		|\tilde{G}_k(y,x_k)|^2_{\tilde{g}_k}
	\lesssim
		U_{k,\max}^{\epsilon - 3}
\]
From this and Lemma~\ref{preliminary-bounds} we get that the second term in \eqref{rep-bound-case2} is bounded by
\[
U_{k,\max}^{\frac{\epsilon}{2} + 1}U_k^{\frac{3}{2} -\alpha} 
\| L_{\tilde{g}_k}(s_k) \|_{C^0_\alpha}
=
\left( U_{k,\max}^{\frac{\epsilon}{2} +1} U_k^{-\frac{3}{2}}\right)
U_k^{3 - \alpha} \| L_{\tilde{g}_k}(s_k)\|_{C^0_\alpha}
\]
When $\epsilon$ is sufficiently small, using the hypothesis on the choice of $U_k$ we have
\[
U_{k,\max}^{\frac{\epsilon}{2} +1} U_k^{-\frac{3}{2}} = o(U_k^{-\frac{1}{8}})
\]
Together with \eqref{rep-bound-case2} and \eqref{L1-over-Sk}, this completes the proof.
\end{proof}

\subsection{Case 3}

It remains to treat the case $w(x_k) \to \infty$ whilst $w(x_k)/U_{k,\max} \to 0$. This means that $x_k$ lies in the intermediate region, between the model and the genuinely hyperbolic part of $X_k$. In particular, $w(x_k) = u(x_k)$. Here, like in Case~1, the metric is very close to hyperbolic. However, the radius on which this holds is not as large and so the same arguments do not work. Instead, we work directly. 

Since $w(x_k)/U_{k,\max} \to 0$ the point $x_k$ lies in a tubular neighbourhood of a component $\Sigma'_k$ of the branch locus:
\[
x_k \in \frac{[u_a, U_{k,\max}) \times S^1 \times \Sigma_k'}{\sim}
\]
(where, as above, $\sim$ denotes the relation collapsing $\{u_a\} \times S^1$ to form a smooth manifold without boundary). We use $(u,\theta) \in [u_a,U_{k,\max})\times S^1$ for the corresponding polar coordinates transverse to $\Sigma_k'$. Recall that in these coordinates the metric $g_k$ is given by
\[
g_k = \frac{\diff u^2}{V(u)} + V(u)l^2 \diff \theta^2 + u^2 h_{\Sigma_k'}
\]
where $l$ is the degree of the branched cover $X_k \to M_k$, $h_{\Sigma_k'}$ denotes the hyperbolic metric on $\Sigma_k'$ and where $V(u) = u^2 - 1 + a \chi u^{-1}$ and $\chi$ is a cut-off function (see the discussion around \eqref{interpolating-V}). When $u$ is large, this is close to the pull-back $h_k$ of the hyperbolic metric from $M_k$:
\begin{equation}
h_k = \frac{\diff u^2}{u^2 - 1} + l^2 (u^2-1) \diff \theta^2 + u^2 h_{\Sigma_k'}
\label{hyperbolic-pulled-back}
\end{equation}
We will use the Green's operator of $h_k$ which, as we now explain, comes with explicit \emph{pointwise} bounds.

First, we work on $\H^4$, with hyperbolic metric $h$. Given $x \in \H^4$, let $G_h(\cdot,x)$ denote the fundamental solution for $L_h$. Here, $L_h$ denotes the linearised Einstein operator in Bianchi gauge relative to $h$, given by~\eqref{linearised-Einstein}. For any smooth compactly supported section $s$ of $S^2T^*\H^4$ we have
\[
s(x) = \int_{\H^4} G_h(x,y)(L_h (s)(y)) \dvol_h(y)
\]
Notice that here we have used the fact that, since $L_h$ is self-adjoint, the Green's operator has the symmetry property $G_h(x,y)=G(y,x)^t$. 

Now we pass to the branched cover. Fix a totally geodesic copy of $\H^2 \subset \H^4$ and write $p \colon \tilde{\H}^4 \to \H^4$ for the $l$-fold cover, branched along $\H^2$. In polar coordinates $(u,\theta,x) \in [1,\infty) \times S^1 \times \H^2$ orthogonal to $\H^2$ (where $u = \cosh(d)$, with $d$ the distance to the copy of $\H^2$), the map $p$ is given by $p(u,\theta, x) = (p, l\theta,x )$. Write $\tilde{h} = p^*h$ for the pull-back of the hyperbolic metric. We set
\[
D= (1,\infty) \times \left(-\frac{\pi}{l}, \frac{\pi}{l}\right) \times \H^2
\]
The restriction of $p$ is an isometry between $(D, \tilde{h})$ and the complement in $\H^4$ of the set $\{ \theta=\pi \}$. It follows that 
\[
G_{\tilde{h}}(x,y) \defeq p^* \circ G_h(p(x),p(y)) \circ p_*
\]
is a fundamental solution for $L_{\tilde{h}}$ over $D$. In other words, if $s$ is a section of $S^2T^*\tilde{\H}^4$ supported in $D$, then for any $x \in D$, 
\begin{equation}
s(x) = \int_D G_{\tilde{h}}(x,y) \left( L_{\tilde{h}} (s)(y) \right) \dvol_{\tilde{g}}(y)
\label{branched-cover-greens-formula}
\end{equation}

We will use cut-off functions to transfer $G_{\tilde{h}}$ to $X_k$. We will support our symmetric 2-tensors in a region of $X_k$, centred on $x_k$ and consisting of a large interval in $u$, a narrow interval in $\theta$ and a large ball in $\Sigma_k'$. The idea is that the rapid decay of $G_{\tilde{h}}$ ensures that the contribution outside of this region to the representation formula will already be too small to matter. 

We now turn to the details. Write $x_k = (u(x_k), 0, z_k) \in [u_a, U_{k,\max})\times S^1 \times \Sigma_k'$. Choose geodesic normal coordinates $(s,\phi)$ on $\Sigma_{k}'$ centred at $z_k$, where $s$ denotes distance to $z_k$ in the hyperbolic metric and $\phi \in S^1$ is the angular coordinate. Then $(u,\theta, s, \phi)$ are a system of coordinates on $X_k$ with $x_k$ corresponding to the point $(u(x_k),0,0,*)$ (the $\phi$-coordinate is not defined at $s=0$). Note that $\Sigma_k \subset M_k$ is totally geodesic and so the injectivity radii satisfy $i(\Sigma_k) \geq i(M_k)$. In particular, the coordinate $s$ is defined for $s \in [0, \sigma]$ for some $\sigma>2$ independent of $k$.

Let $M>0$ be large, to be chosen later. We define three cut-off functions as follows:
\begin{itemize}
\item 
Let $\hat\eta_1 \colon \R \to \R_{\geq0}$ be smooth with $\hat\eta_1(v) = 1$ for $v \in [-\log (u(x_k)) - M/2, \log (u(x_k)) + M/2]$ and with $\hat\eta_1(v) = 0$ for $v \in [ - \log(u(x_k)) - M, \log(u(x_k))+M]^c$. We can choose $\hat\eta_1$ with
\[
|\hat\eta_1'| + |\hat\eta_1''| \leq \frac{C}{M}
\]
for some positive constant $C$ independent of $M$. 

We set $\eta_1(u,\theta,s,\phi) = \hat{\eta}_1( \log u)$. By \eqref{hyperbolic-pulled-back}, the derivatives of $\eta_1$ using the pulled-back hyperbolic metric $h_k$ satisfy
\[
|\nabla \eta_1(y)|_{h_k} + |\nabla^2 \eta_1(y)|_{h_k} \leq \frac{C'}{M}
\]
for any $y$, where $C'$ is independent of $M$ and $k$ (and $\nabla$ is taken for $h_k$). 

\item
Let $\hat{\eta}_2 \colon \left[-\frac{\pi}{l}, \frac{\pi}{l}\right] \to \R_{\geq0}$ be smooth with $\hat{\eta}_2(\theta) = 1$ for $\theta \in \left[-\frac{\pi}{2l},\frac{\pi}{2l}\right]$ and $\hat{\eta}_2(\theta)=0$ for $\theta = \left[-\frac{3\pi}{4l},\frac{3\pi}{4l}\right]^c$. We choose $\hat{\eta}_2$ so that 
\[
|\hat{\eta}_2'| + |\hat{\eta}_2''| \leq C
\]

We write $\eta_2(u,\theta,s,\phi) = \hat{\eta}_2(\theta)$. Again using $h_k$ to define derivatives, we have that for any $y$,
\[
|\nabla \eta_2(y)|_{h_k} + |\nabla^2 \eta_2(y)|_{h_k} \leq \frac{C'}{ \sqrt{u(y)^2 -1 }}
\]
where $C'$ is independent of $y$, $M$ and $k$ (but does depend on $l$). 

After applying the first cut-off $\eta_1$, we will be left considering only those points $y$ with $e^{-M} u(x_k) \leq u(y) \leq e^Mu(x_k)$. Now, by assumption, $u(x_k ) \to \infty$ so taking $k$ large enough (depending on $M$) we can ensure that for these points $y$, $u(y) \geq \sqrt{M^2 + 1}$. This gives
\[
|\nabla \eta_2(y)|_{h_k} + |\nabla^2 \eta_2(y)|_{h_k} \leq \frac{C'}{M}
\]
\item
Let $\hat{\eta}_3 \colon \R_{\geq0} \to \R_{\geq0}$ be smooth with $\hat\eta_3(s) = 1$ for $s \in[0,\sigma -2]$ and $\hat\eta_3(s) = 0$ for $s\in [\sigma -1, \infty)$. Put $\eta_3(u,\theta,s,\phi) = \hat{\eta}_3(s)$. Again we have
\[
|\nabla \eta_3(y)|_{h_k} + |\nabla^2 \eta_3(y)|_{h_k} \leq\frac{C'}{u(y)} \leq \frac{C'}{M}
\]
for those $y$ with $e^{-M}u(x_k) \leq u(y)\leq  e^Mu(x_k)$ as long as $k$ is large enough (depending on $M$).
\end{itemize}

We now set $\eta \colon X_k \to \R_{\geq0}$ to be $\eta = \eta_1 \eta_2 \eta_3$. Although each $\eta_i$ is only defined in the coordinate patch where the coordinates $(u,\theta, s, \phi)$ are valid, the product is supported in the subset of $[u_a, U_{k,max}) \times S^1 \times {\Sigma_k}^{'}$ defined by
\begin{equation}
A_k 
	=
		\left\{ e^{-M}u(x_k) \leq u \leq e^Mu(x_k)\right\} \times 
		\left\{ -\frac{3\pi}{4l} \leq \theta \leq \frac{3\pi}{4l} \right\} \times
		\left\{ s \leq  \sigma -1 \right\} \times
		S^1
\label{def-Ak}
\end{equation}
and $\eta(y)$ vanishes for $y$ outside this region. 

Recall that our hypothesis provided a sequence $\tilde{g}_k$ of metrics on $X_k$ satisfying~\eqref{metrics-converge} and a sequence $s_k$ of symmetric bilinear forms with $L_{\tilde{g}_k}(s_k)$ supported in $S_k$ and satisfying~\eqref{bound-to-be-violated}. We write $\tilde{s}_k = \eta s_k$, which is supported in $A_k$. For this section only we denote by $L_{h_k}$ the linearised Ricci operator in the Bianchi gauge with respect to $h_k$ itself, acting on sections compactly supported in $D$. By \eqref{linearised-Einstein} it is a rough laplacian plus zeroth-order terms, and direct computation gives
\begin{equation}
L_{h_k} (\tilde{s}_k)
	=
		\eta L_{h_k} (s_k)
		-
		 \nabla_{\nabla \eta} s_k
		+
		+ \frac12 (\Delta_{h_k} \eta) s_k
\label{L-cut-off}
\end{equation}
where here $\nabla$ is taken with respect to $h_k$. The last two terms here are supported only where $\nabla \eta$ is non-zero, i.e., in the region $B_k$ given by \begin{eqnarray}
B_k
	&=&
		B_k' \cap A_k \nonumber\\
B'_k
	&=&
		\left\{ e^{-M}u(x_k) \leq u \leq e^{-M/2}u(x_k)\right\} 
		\cup
		\left\{ e^{M/2}u(x_k) \leq u \leq e^Mu(x_k)\right\}
		\nonumber\\
	&~&\quad\quad	
		\cup
		\left\{- \frac{3\pi}{4l} \leq \theta \leq - \frac{\pi}{2l} \right\}
		\cup
		\left\{ \frac{\pi}{2l} \leq \theta \leq \frac{3\pi}{4l}\right\}
		\cup
		\left\{ \sigma - 2 \leq s \leq \sigma -1\right\}
\label{def-Bl}
\end{eqnarray}

We want to compare \eqref{L-cut-off} to $L_{\tilde{g}_k}(s_k)$. We begin with the last two terms. By definition of the $\eta_i$, we have that for any $y \in B_k$, 
\[
\left|
- \nabla_{\nabla \eta} s_k + \frac12 (\Delta_{h_k} \eta) s_k
\right|_{h_k}(y)
\leq \frac{C}{M} \big( | s_k|_{h_k}(y)  + |\nabla s_k|_{h_k}(y)\big)
\]
for some $C$ independent of $k$ or $M$, and for all $k$ sufficiently large. Since $u(x_k) \to \infty$, we have that for any $m \in \N$, as $k \to \infty$,
\begin{equation}\label{hk-close-to-gk}
\| g_k - h_k \|_{C^m(A_k)} \to 0
\end{equation}
This means that for $y \in B_k$,
\[
\left|
- \nabla_{\nabla \eta} s_k + \frac12 (\Delta_{h_k} \eta) s_k
\right|_{h_k}(y)
\leq \frac{C}{M} \big( | s_k|_{g_k}(y)  + |\nabla s_k|_{g_k}(y)\big)
\]
where on the right-hand side $\nabla$ is now defined using $g_k$. From here, and by~\eqref{norms-on-balls} we have
\begin{equation}
\left|
- \nabla_{\nabla \eta} s_k + \frac12 (\Delta_{h_k} \eta) s_k
\right|_{h_k}(y)
\leq
\frac{C}{M} u(y)^{-\alpha}\| s_k \|_{C^1_\alpha}
\label{pointwise-bound-two-terms}
\end{equation}
Now, by the weighted elliptic estimate, Lemma~\ref{weighted-elliptic}, for the metric $\tilde{g}_k$, together with the hypothesis~\eqref{bound-to-be-violated}, we have
\[
\| s_k \|_{C^{m+2,\eta}_\alpha} 
	\lesssim 
		\| s_k \|_{C^0_\alpha}
\]
Together with~\eqref{pointwise-bound-two-terms}, this gives
\begin{equation}
\left|
- \nabla_{\nabla \eta} s_k + \frac12 (\Delta_{h_k} \eta) s_k
\right|_{h_k}(y)
\leq 
\frac{C}{M} u(y)^{-\alpha}\| s_k \|_{C^0_\alpha}
\label{two-errors-alpha-bound}
\end{equation}

Meanwhile, with an eye on the first term in~\eqref{L-cut-off},  we note that, by~\eqref{metrics-converge}, $\tilde{g}_k - g_k$ is small in $C^{m+2,\eta}_\alpha$. Together with~\eqref{hk-close-to-gk}, and similar reasoning as led to~\eqref{two-errors-alpha-bound},  we see that for any $y \in A_k$,
\begin{equation}
\left| \left(L_{h_k} - L_{\tilde{g}_k}\right) (s_k) \right|_{g_k}(y)
	=
		o \left( u(y)^{-\alpha} \| s_k \|_{C^0_\alpha}\right)
\label{Lhk-bound}
\end{equation}

We are now in a position to apply the representation formula~\eqref{branched-cover-greens-formula} for tensors compactly supported in the region $D \subset \tilde{\H}^4$. Since $\tilde{s}_k$ is supported in a region on which $h_k$ is isometric to $(D, \tilde{h})$ we can apply this formula to $\tilde{s}_k$. For clarity, we will only consider from now on norms defined with respect to $\tilde{h}$.

 When taken together with \eqref{two-errors-alpha-bound} and~\eqref{Lhk-bound} we obtain the following (recall that $S_k$ is the gluing region, in which we assume $L_{\tilde{g}_k}(s_k)$ is supported):
\begin{multline}
|\tilde{s}_k(x_k)|_{g_k}
	\lesssim
		\| L_{\tilde{g}_k} (s_k)\|_{C^0_\alpha} \cdot \int_{A_k \cap S_k} 
			u(y)^{-\alpha} |G_{\tilde{h}}(x_k,y)|_{\tilde{h}} \dvol_{\tilde{h}}(y) \\
		+
		o\left(
		 \| s_k\|_{C^0_\alpha} \cdot \int_{A_k}
		 u(y)^{-\alpha} |G_{\tilde{h}}(x_k,y)|_{\tilde{h}} \dvol_{\tilde{h}}(y)	 
		 \right)\\
		 +
		 O\left(  
		 \frac{1}{M} \| s_k \|_{C^0_\alpha} \cdot \int_{B_k}
		 u(y)^{-\alpha}|G_{\tilde{h}}(x_k,y)|_{\tilde{h}} \dvol_{\tilde{h}}(y)
		 \right)			
\label{rep-formula-in-pieces}
\end{multline}
The crux then is to control the integrals appearing in~\eqref{rep-formula-in-pieces}.

\begin{lemma}\label{Ak-integral-bound}
Let $0 < \alpha < 3$. Then
\[
\int_{A_k} u(y)^{-\alpha} |G_{\tilde{h}}(x_k,y)|_{\tilde{h}} \dvol_{\tilde{h}}(y)
	\lesssim
		u(x_k)^{-\alpha}
\]
\end{lemma}
\begin{proof}
The Green's operator $G_{\tilde{h}}(x_k,y)$ has a pole at $y=x_k$ which we take care of first. By definition~\eqref{def-Ak} of $A_k$, there is a radius $\rho >0$ depending only on the degree $l$ of the cover (and  in particular independent of $k$ and $M$), such that $B_{\tilde{h}}(x_k,\rho) \subset A_k$. Since $u$ is of controlled growth on this ball, for any $y \in B_{\tilde{h}}(x_k,\rho)$, we have $u(y)^{-\alpha} \lesssim u(x_k)^{-\alpha}$. Meanwhile, the Green's operator of hyperbolic space is $L^1$ near the pole, so that
\[
\int_{B_{\tilde{h}}(x_k,\rho)} |G_{\tilde{h}}(x_k,y)|_{\tilde{h}} \dvol_{\tilde{h}}(y)
=
\int_{B_h(p,\rho)} |G_{h}(p,q)|_{h} \dvol_{h}(q)
=
C
\]
is independent of $k$. It follows that 
\[
\int_{B_{\tilde{h}}(x_k,\rho)} u(y)^{-\alpha} |G_{\tilde{h}}(x_k,y)|_{\tilde{h}} \dvol_{\tilde{h}}
\lesssim 
u(x_k)^{-\alpha}
\]

Away from the pole we will use the pointwise exponential decay of $G_h$. For the hyperbolic Green's function of $L_h$, for any $\rho >0$ there is a constant $C$ such that for any $p,q \in \H^4$ with $d_h(p,q) \geq \rho$, we have
\[
|G_h(p,q)|_h \leq C e^{-3 d_h(p,q)}
\]
See for instance Biquard \cite{biquard2}, Proposition $1.2.2$. It follows that for $y \in A_k \setminus B_{\tilde{h}}(x_k, \rho)$, we have
\begin{equation}
|G_{\tilde{h}}(x_k, y)|_{\tilde{h}} \leq Ce^{-3 d_{h}(p(x_k), p(y))}
\label{exponential-decay}
\end{equation}
We express this distance function in the coordinates $(u, \theta, s, \phi)$. We parameterize the one-sheeted hyperboloid in coordinates $(u,\theta,s,\phi)$ and use the well-known explicit expression of the hyperbolic distance. With $x_k = (u(x_k), 0, 0, *)$ and $y = (u,\theta, s, \phi)$ the bound~\eqref{exponential-decay} then reads
\begin{equation}
|G_{\tilde{h}}(x_k,y)|_{\tilde{h}}
\leq 
C\left( u(x_k) u \cosh(s) - \sqrt{(u(x_k)^2 - 1)(u^2 -1)} \cos(l\theta)\right)^{-3}
\label{pointwise-bound-Ghk}
\end{equation}
Now since the volume form is
\[
\dvol_{\tilde{h}} = lu^2 \sinh(s) \diff u \diff \theta \diff s \diff \phi
\]
we see that
\begin{multline*}
\int_{A_k\setminus B_{\tilde{h}}(x_k, \rho)}
u(y)^{-\alpha} |G_{\tilde{h}}(x_k, y)|_{\tilde{h}} \dvol_{\tilde{h}}(y)
	\lesssim\\
		\int_{A_k} \left(  u(x_k) u \cosh(s) - \sqrt{(u(x_k)^2 - 1)(u^2 -1)} \cos(l\theta)\right)^{-3}
		lu^{2-\alpha} \sinh(s) \diff u \diff \theta \diff s \diff \phi
\end{multline*}
The right-hand side is easily computed with standard integrals. In particular:
\begin{align}
\int_{A_k \setminus B_{\tilde{h}}(x_k, \rho)}
u(y)^{-\alpha} |G_{\tilde{h}}(x_k, y)|_{\tilde{h}} \dvol_{\tilde{h}}(y)
	& \lesssim
		\int_{e^{-M}u(x_k)}^{e^Mu(x_k)}
		\frac{u^{2-\alpha} \diff u}{\left( u(x_k)^2+ u^2 \right)^{3/2}}\nonumber\\
	&=
		u(x_k)^{-\alpha} \int_{e^{-M}}^{e^M}
		\frac{u^{2-\alpha} \diff u}{\left( 1 + u^2 \right)^{3/2}}
		\label{u-integral}
\end{align}
This last integral converges both at $0$ and $\infty$ when $M \to \infty$, provided $0 < \alpha < 3$, completing the proof of the Lemma.
\end{proof}

Next we control the integral over the region $B_k$, where $\nabla \eta$ is non-zero. We write $B_k = B_k^u \cup B_k^\theta \cup B_k^s$ where
\begin{align*}
B_k^u 
	&= 
		A_k \cap \left\{ e^{-M}u(x_k) \leq u \leq e^{-M/2}u(x_k)\right\}
		\cup
		A_k \cap \left\{ e^{M/2}u(x_k) \leq u \leq e^Mu(x_k)\right\}\\
B_k^\theta
	&=
		A_k \cap \left\{-\frac{3\pi}{4l} \leq \theta \leq -\frac{\pi}{2l} \right\}
		\cup
		A_k \cap \left\{ \frac{\pi}{2l} \leq \theta \leq \frac{3\pi}{4l}\right\}\\
B_k^s
	&=
		A_k\cap \left\{ \sigma -2 \leq s \leq \sigma -1 \right\}		
\end{align*}

\begin{lemma}\label{Bku-integral-bound}
Let $0 < \alpha < 3$. There is a function $\epsilon \colon \R \to \R_+$ with $\epsilon(M) \to 0$ as $M \to \infty$ such that
\[
\int_{B_k^u} u(y)^{-\alpha}|G_{\tilde{h}}(x_k,y)|_{\tilde{h}} \dvol_{\tilde{h}}(y) 
	\lesssim
		\epsilon(M) u(x_k)^{-\alpha}
\]
\end{lemma}

\begin{proof}
By the same arguments as led to~\eqref{u-integral}, the integral over $B_k^u$ is bounded by
\[
\left( 
\int_{e^{-M}}^{e^{-M/2}} \frac{u^{2 -\alpha}\diff u}{\left( 1 + u^2 \right)^{3/2}}
+
\int_{e^{M/2}}^{e^M} \frac{u^{2 -\alpha}\diff u}{\left( 1 + u^2 \right)^{3/2}}
\right) u(x_k)^{-\alpha}
\]
Since the integral of $u^{2-\alpha}(1+u^2)^{-3/2}$ over the whole of $\R_+$ is finite, these two tails must tend to zero as $M \to \infty$.
\end{proof}

\begin{lemma}\label{Bktheta-integral-bound}
Let $\alpha>0$. Then, as $k\to \infty$, 
\[
\int_{B_k^\theta} u(y)^{-\alpha}|G_{\tilde{h}}(x_k,y)|_{\tilde{h}}\dvol_{\tilde{h}}(y)
	=
		o\left( u(x_k)^{-\alpha} \right)
\]
\end{lemma}

\begin{proof}
By the pointwise control of $G_{\tilde{h}}$ in~\eqref{pointwise-bound-Ghk}, we have
\begin{align*}
\int_{B_k^\theta} u(y)^{-\alpha}|G_{\tilde{h}}(x_k,y)|_{\tilde{h}} \dvol_{\tilde{h}}(y)
	&\lesssim
		u(x_k)^{-3} 
			\int_{u=e^{-M}u(x_k)}^{e^Mu(x_k)}
			\int_{s=0}^\infty
			u^{-1-\alpha}\cosh(s)^{-3}\sinh(s) \diff u\diff s\\
	&\lesssim
		e^{\alpha M}u(x_k)^{-3-\alpha}\\
	&=
		o\left( u(x_k)^{-\alpha} \right)	\qedhere
\end{align*}
\end{proof}

\begin{lemma}\label{Bks-integral-bound}
Let $\alpha >0$. Then, as $k \to \infty$
\[
\int_{B_k^s} u(y)^{-\alpha}|G_{\tilde{h}}(x_k,y)|_{\tilde{h}}\dvol_{\tilde{h}}(y)
	=
		o\left( u(x_k)^{-\alpha} \right)
\]
\end{lemma}
\begin{proof}
By the pointwise control of $G_{\tilde{h}}$ in~\eqref{pointwise-bound-Ghk}, we have
\begin{align*}
\int_{B_k^s} u(y)^{-\alpha}|G_{\tilde{h}}(x_k,y)|_{\tilde{h}}\dvol_{\tilde{h}}(y)
	&\lesssim
		u(x_k)^{-3}
			\int_{e^{-M}u(x_k)}^{e^Mu(x_k)}
			u^{-1-\alpha} \diff u\\
	&\lesssim
		e^{\alpha M}u(x_k)^{-3-\alpha}\\
	&=
		o\left( u(x_k)^{-\alpha} \right)	\qedhere
\end{align*}
\end{proof}

We are finally ready to prove the estimate which contradicts~\eqref{bound-to-be-violated}.
\begin{proposition}
Let $\alpha >0$. As $k \to \infty$, 
\[
\| s_k \|_{C^0_\alpha} \lesssim \| L_{\tilde{g}_k}(s_k)\|_{C^{0,\eta}_\alpha}
\]
\end{proposition}
\begin{proof}
By~\eqref{rep-formula-in-pieces} and Lemmas~\ref{Ak-integral-bound}, \ref{Bku-integral-bound}, \ref{Bktheta-integral-bound} and~\ref{Bks-integral-bound} we have that at the point $x_k$ where $u(x)^\alpha |s_k(x)|_{g_k}$ is maximal,
\[
u(x_k)^\alpha|s_k(x_k)|_{g_k}
	\lesssim
		\|L_{\tilde{g}_k}(s_k)\|_{C^0_\alpha}
		+
		o \left( \|s_k\|_{C^0_\alpha} \right)
		+
		\left( \frac{\epsilon(M)}{M} + o(1)  \right) \|s_k\|_{C^0_\alpha}
\]
where $\epsilon(M) \to 0$ as $M \to \infty$. We now take $M$ very large (but fixed) so that the $\|s_k\|_{C^0_\alpha}$ terms on the right-hand side can be absorbed on the left.
\end{proof}

This Proposition shows that Case~3 cannot actually have occurred after all. Since we have already ruled out the other possibilities, the proof of Theorem~\ref{main-theorem} is complete.

\appendix

\section{Technical results on Green's operators}\label{Green}

For the purpose of this Appendix, we will denote by $(M_k,g_k)$ the sequence of complete non-compact negatively curved  Riemannian four-manifolds investigated respectively in \S\ref{case-1-section} and~\S\ref{case-2-section}. More precisely, $(M_k,g_k)$ will denote:
\begin{itemize}
\item either  $(\mathbb{H}^4, g_k)$, where $g_k$ is the extension of the metric $g_k$ (originally constructed in Proposition~\ref{prop-approx-solution}) to $\mathbb{H}^4$. The extension process is described after~\eqref{global-close-to-hyperbolic};
\item or $(Y_k, g_k)$, where $Y_k$ is defined in \S\ref{case-2-section} and $g_k$ is again the extension of the approximate Einstein metric to $Y_k$, interpolated with the hyperbolic metric at large distances.
\end{itemize}
We will also let $(\tgk)$ be another family of Riemannian metrics on $M_k$, which coincides with $g_k$ outside of a compact set $K_k$ of $M_k$ (possibly depending on $k$) and which satisfies 
\begin{equation} 
\label{tgkclosedtogk}
\| \tgk - g_k \|_{C^{m,\eta}(M_k) }= o(1)
\end{equation}
for some $m \ge 3$ and $0 < \eta < 1$. In the first case where $M_k = \mathbb{H}^4$ we will also assume that  $\| \tgk - g_{\mathrm{hyp}} \|_{C^{m,\eta}(\mathbb{H}^4) }= o(1)$. These assumptions are met when $\tgk$ is the metric considered in Section~\ref{case-1-section} and~\ref{case-2-section}. In both cases, these assumptions imply that $i_{\tgk}(M_k)$ is uniformly bounded from below by a positive number.

The proof of Theorem~\ref{key-estimate} relies on suitable representation formulae for $L_{\tgk}$. In this Appendix we construct a fundamental solution for $L_{\tgk}^*$ on $M_k$, describe its behavior around the singularity and prove a global $L^2$ estimate for it. In the following we let $S = S_{(2,0)}(T^* M_k)$ be the bundle of symmetric bilinear forms on $T^*M_k$. We prove the following result:

\begin{proposition} \label{propGreen}
Let $x \in M_k$. There exists a fundamental solution $G_k(\cdot,x)$ for $L_{\tgk}^*$in the following sense: it is a section of the bundle $y \mapsto Hom(S_x, S_y)$ which is in $C^{1,\eta}(M_k \backslash \{x\})$ and satisfies the following distributional equation:
\[ L_{\tgk}^* \Big( G_k(\cdot, x) \big(f^{ij}(x) \big) \Big) = \delta_x \cdot f^{ij}(x), \]
for any $1 \le i,j \le n$, where $f^{ij}$ is a local orthonormal basis of $S$ around $x$ and parallel at $x$. 
Also, $G_k$ satisfies that:
\begin{equation}
\label{prop2}
\int_{M_k \backslash \{d_{\tgk}(x,y) \le r \} } 
	\left( 
	|G_k(y,x)|_{\tgk}^2 + |\nabla_{\tgk} G_k(y,x)|_{\tgk}^2 
	\right) \dvol_{\tgk}(y) 
		\leq 
			C(r), 
\end{equation}
and that:
\begin{equation}
\label{prop0}
\int_{B_{\tgk}(x,r)}  
	\left( 
	|G_k(y,x)|_{\tgk(y)} +  |\nabla_{\tgk} G_k(y,x)|_{\tgk} 
	\right) \dvol_{\tgk}(y) 
		\leq C(r)
\end{equation}
for some positive constant $C(r)$ independent of $x$ and $k$ and depending only on $r$.
\end{proposition}

As a consequence of Proposition \ref{propGreen} we get that, for any compactly supported bilinear form $h$ of class $C^2$ in $M_k$:
\begin{equation}
\label{formulerepcool}
h(x) = \int_{M_k} G_k(y,x)^t \left( L_{\tgk}(h)(y) \right) \dvol_{\tgk}(y),
\end{equation}
where $G_k(y,x)^t$ is the adjoint of the linear map $G_k(y,x): S_x \to S_y$.

\begin{proof}
For a fixed Riemannian metric, the arguments used here are standard and are an adaptation of the constructions in \cite{aubin2} (see also~\cite{robert}). The point is to ensure that the constants in \eqref{prop2} and \eqref{prop0} are independent of $\tgk$. We sketch the relevant parts of the proof.

An explicit expression for the operator $L_{\tgk}^*$, that we will not write here, can be obtained from formula~\ref{linearised-Einstein-general}. In particular, in a neighbourhood of $x$, $L_{\tgk}^*$ looks like, at first order, a constant-coefficients elliptic system $L_k$ in $\R^{\frac{n(n+1)}{2}}$ satisfying, by \eqref{tgkclosedtogk}, the Legendre condition. Fundamental solutions for such systems, whose norm behaves as $d_{\tgk}(x,\cdot)^{2-n}$ around $x$, are known to exist (see for instance~\cite{dolzmann-muller}). Let $0 <\delta_0 < i_{\tgk}(M_k)$. If $H_k$ is a fundamental solution of $L_k$ we define, for $1 \le i,j \le n$ and for $y \in B_{\tgk}(x, \delta_0) \backslash \{x\}$:
\begin{equation}
\label{defG0}
 G_{0,k}^{ij} \big(y ,x\big) := \chi \big( d_{\tgk}(x,y) \big) H_{k}(\eta)_{pq}^{ij}  f^{pq} (y), 
 \end{equation}
where $\eta$ is such that $y = \exp_x (\eta)$, $\exp_x$ is the exponential map at $x$ for $\tgk$ and $\chi$ is a cut-off function in $\R_+$ such that $\chi \equiv 1$ in $[0,\frac12 \delta_0]$ and $\chi \equiv 0 $ in $[\delta_0, +\infty)$. Equivalently, we denote by $G_{0,k}(y,x)$ the element of $\textrm{Hom}(S_x, S_y)$ such that $G_{0,k}(y,x)(f^{ij}(x)) = G_{0,k}^{ij}(y,x)$.  This is now a global section of $S$, smooth on $M_k \backslash \{x\}$, singular at $x$ and compactly supported in $B_{\tgk}(x, \delta_0)$. Straightforward computations using \eqref{tgkclosedtogk} show that the following holds, in a distributional sense:
\begin{equation} 
\label{dist1}
 L_{\tgk}^* G_{0,k}^{ij}(\cdot, x) = \delta_x f^{ij}(x) + R_{0,k}^{ij} (\cdot,x),
\end{equation}
where $R_{0,k}^{ij}(y,x) = R_{0,k}(y,x) \big(f^{ij}(x) \big)$ and $R_{0,k}(\cdot,x)$ is a singular section of $S$, supported in $B_{\tgk}(x, \delta_0)$, satisfying
\begin{equation}
\label{defR0}
|R_{0,k}(y,x)|_{\tgk} \le C d_{\tgk}(x,y)^{1-n}
\end{equation}
for some positive $C$ independent of $x$ and $k$.  As a consequence of \eqref{dist1} we get that for any $h \in S$ compactly supported and of class $C^2$ one has:
\begin{equation} 
\label{rep}
 h(x) = 
 	\int_{M_k} G_{0,k}(y,x)^t 
 	\left( L_{\tgk}(h)(y) \right) \dvol_{\tgk}(y) 
	-  
	\int_{M_k} R_{0,k}(y,x)^t \left( h(y) \right) \dvol_{\tgk}(y)
\end{equation}

The iterative construction now goes as follows: for any $p \ge 0$ we let
\begin{equation}
\label{defGn1}
 G_{p+1,k}^{ij} (y,x) = G_{0,k}^{ij} (y,x) -  \sum_{q=0}^{p} \int_{M_k} G_{0,k}(y,z) \circ R_{q,k}(z,x)  (f^{ij}(x)) \textrm{dvol}_{\tgk}(z),
\end{equation}
and 
\[
R_{p+1,k}^{ij} (y,x) = - \int_{M_k} R_{0,k}(y,z) \circ R_{p,k} (z,x) (f^{ij}(x)) \textrm{dvol}_{\tgk}(z) 
\]
By \eqref{dist1}:
\[ 
L_{\tgk}^* G_{p+1,k}^{ij} (\cdot, x) = \delta_x f^{ij}(x) + R_{p+1,k}^{ij} (\cdot,x). 
\]
An iterative use of Giraud's lemma (see \cite{hebey}, Lemma~7.5) with \eqref{defG0} and \eqref{defR0} shows that 
\begin{equation}
\label{controleRn1} 
\left| R_{n-1,k}^{ij} (y,x) \right|_{\tgk}  \le C | \ln d_{\tgk}(x,y)| 
\end{equation}
and that 
\begin{equation} 
\label{controleGn1}
|G_{p+1,k} (y,x)|_{\tilde{g}_k} \le C d_{\tgk}(x,y)^{2-n},
\end{equation}
for some positive $C$ independent of $k,x,$ and $y$, for any $0 < d_{\tgk}(x,y) \le M \delta_0$ for some $M > 0$. Note also that by construction $R_{n-1, k}^{ij} (\cdot, x)$ is compactly supported in $M_k$. It remains to solve for the remainder term. We claim that there exists a unique bilinear form $N_k^{ij}(\cdot,x)$ in $H^1(M_k)$ satisfying:
\[ 
L_{\tgk}^* N_k^{ij}(\cdot,x) = - R_{n-1,k}^{ij} (\cdot,x)  
\]
and  
\begin{equation} 
\label{L2Nk}
 \Vert N_k^{ij} (\cdot, x) \Vert_{H^1(M_k)} \le C 
\end{equation}
for some positive $C$ independent of $x$ and $k$. The existence of $N_k^{ij}(\cdot,x) $ and the bound \eqref{L2Nk} follow from~\ref{linearised-Einstein-general} and from \eqref{tgkclosedtogk}, which show that $L_{\tgk}^*$ is $o(1)$-close, in operator norm, to an injective self-adjoint elliptic operator. By \eqref{controleRn1} and since $R_{n-1,k}^{ij} (\cdot,x)$ is compactly supported, it in $L^p(M_k)$ for all $p \ge 1$, so standard elliptic regularity theory with \eqref{tgkclosedtogk} shows that $N_k^{ij}(\cdot, x) $ is of class $C^{1, \eta}$ in $M_k$ for all $0 < \eta < 1$ and that 
\begin{equation}
\label{C1Nk}
\|N_k^{ij}(\cdot, x)\|_{C^{1, \eta}(B_{\tgk}(x, 2 \delta_0))} \le C. 
\end{equation}
It remains to define $ G_{k} (y,x) = G_{n-1,k} (y, x) + N_k(y,x)$, where $G_{n-1,k}(y,x)$ is given by \eqref{defGn1}. There holds then, for any $1 \le i,j \le n$:
\[ L_{\tgk}^* G_k^{ij}(\cdot, x) = \delta_x f^{ij}(x), \]
and \eqref{prop2} and \eqref{prop0} follow from \eqref{controleGn1}, \eqref{L2Nk} and \eqref{C1Nk}. 
\end{proof}

We also remark that by standard elliptic theory applied to the operator $L_{\tgk}^*$, and as a consequence of \eqref{tgkclosedtogk}, we also get uniform (in $k$) $C^{m,\eta}$ bounds on $G_k(\cdot,x)$ in every fixed compact set of $M_k \backslash \{x\}$. 

\bibliographystyle{amsplain}
\bibliography{Einstein-bibliography}

\end{document}